\documentclass[10pt]{article}

\usepackage{amsmath}
\usepackage{amsfonts}
\usepackage{tikz-cd}
\usepackage[dvipsnames]{xcolor}
\definecolor{rosso}{RGB}{183,42,58}
\definecolor{SNSBlue}{RGB}{0,114,142}
\definecolor{uerrelle}{RGB}{200,150,120}
\usepackage[colorlinks=true,hyperindex,linkcolor=rosso, pagebackref=false, citecolor=SNSBlue, urlcolor=rosso, pdfpagelabels]{hyperref}
\usepackage{utfsym}
\usepackage{geometry}
\usepackage{calc}
\usepackage{amsthm}
\usepackage{mathrsfs}
\usepackage{amssymb}

\newcommand{\onto}[2][]{\xrightarrow[#1]{#2}\mathrel{\mkern-14mu}\rightarrow}
\newcommand\xto[2]{\xrightarrow[#1]{#2}}
\DeclareMathOperator{\im}{Im}
\DeclareMathOperator{\id}{id}
\DeclareMathOperator{\sub}{\subseteq}

\newcommand\Z{\mathbb{Z}}
\newcommand\Q{\mathbb{Q}}

\newcommand{\Fp}{\mathbb{F}_p}

\newcommand{\Zp}{\mathbb{Z}_p}
\newcommand{\Qp}{\mathbb{Q}_p}

\newcommand\calF{\mathcal{F}}

\newcommand\calG{\mathcal{G}}

\newcommand\pmeninf{p^{-\infty}}

\newcommand{\Fpbar}{\overline{\mathbb{F}}_p}

\DeclareMathOperator{\Rep}{Rep}
\newcommand{\Hom}{\mathrm{Hom}}
\DeclareMathOperator{\Aut}{\underline{\mathcal{A}\mathit{ut}}}
\newcommand{\End}{\underline{\mathcal{E}\mathit{nd}}}
\DeclareMathOperator{\sHom}{\underline{\mathcal{H}\mathit{om}}}
\DeclareMathOperator{\Spec}{Spec}
\newcommand{\Spf}{\mathrm{Spf}}
\newcommand{\Acr}{\mathbb{A}_\textnormal{cris}}
\newcommand{\Isoc}{\textnormal{Isoc}}
\DeclareMathOperator{\cris}{Cris}

\newcommand{\GL}{\textnormal{GL}}

\newcommand{\W}{W}
\newcommand{\Lie}{\textnormal{Lie}}
\newcommand{\C}{\mathbb{C}}

\newcommand{\alg}{\textnormal{alg}}

\newcommand{\defsus}{\textnormal{\textbf{Def}}_\textnormal{sus}}

\newcommand{\calA}{\mathcal{A}}
\DeclareMathOperator{\Isom}{\underline{\mathcal{I}\!\mathit{som}}}
\DeclareMathOperator{\Sym}{\textnormal{Sym}}

\newcommand\Hcris[1]{H^{#1}_{\textnormal{cris}}}

\newtheorem{lem}{Lemma}[section]
\newtheorem{prop}[lem]{Proposition}
\newtheorem{cor}[lem]{Corollary}
\newtheorem{thm}[lem]{Theorem}

\theoremstyle{definition}
\newtheorem{deffa}[lem]{Definition}
\newtheorem{ex}[lem]{Example}
\newtheorem{rem}[lem]{Remark}
\newtheorem{setting}[lem]{Setting}

\renewcommand{\qed}{\hfill\square\vspace{.5em}}
\newcommand{\unit}[1]{\vspace{.5em}\noindent{\bf #1}}

\title{Monodromy of Tate-Linear varieties}
\date{\today}
\author{Livia Grammatica}
\begin{document}
\maketitle

\begin{abstract}
    We answer a question of Chai about the rigidity of Tate-linear subvarieties of $\mathcal{A}_g$ over $\overline{\mathbb{F}}_p$: we show that that a smooth subvariety of a central leaf of $\mathcal{A}_g$ which is Tate-linear at one point is Tate-linear at every point. This is achieved by a careful analysis of the $p$-adic monodromy groups of Tate-linear varieties, building on recent work of D'Addezio and van Hoften.
\end{abstract}

\tableofcontents

\section{Introduction}

Let $p$ be a prime and let $k=\Fpbar$. In this paper we study a notion of linearity for smooth subvarieties of $\calA_{g,k}$, which can be thought as an algebraic counterpart to the notion of totally geodesic subvarieties of $\calA_{g,\C}$. The first algebraic definition of linearity was introduced by Moonen in his PhD thesis \cite{moonen}, drawing inspiration from work of Noot. Recall that if $x$ is an ordinary $k$-point of $\mathcal{A}_{g,\W(k)}$, the theory of Serre--Tate canonical coordinates endows the deformation space $\mathcal{A}_{g,\W(k)}^{/x}$ with the structure of a formal torus over $\W(k)$. Moonen then calls a subvariety $Z\subseteq\mathcal{A}_{g,\W(k)}$ formally quasi-linear at $x$ if 
each component of the completion $Z^{/x}$ is a formal subtorus of $\mathcal{A}_{g,\W(k)}^{/x}$. 

The two authors' interest in this notion lies in the close relationship between formally quasi-linear subvarieties of $\calA_g$ and Shimura subvarieties of $\calA_g$. Noot first showed that the formal completion at an ordinary point 
of (the naïve integral model of) a Shimura subvariety of a Siegel modular variety is formally quasi-linear \cite[Theorem 3.7]{noot}. Moonen then proved a converse of this statement: each subvariety of $\calA_g$ defined over a number field, which is formally quasi-linear at an ordinary point, is in fact a Shimura subvariety of $\calA_g$ \cite[Theorem 4.5]{moonen}. This result shows the remarkable rigidity of subvarieties which are formally linear at an ordinary point. 

In this paper we study a notion of linearity in positive characteristic, called Tate-linearity, which is due to Chai--Oort \cite{chaisustained}. It applies to the completion of $\calA_{g,k}$ at points belonging to arbitrary Newton strata and can be seen as an extention of the theory Serre-Tate canonical coordinates to non-ordinary points. The starting point is that for $x\in\calA_g(k)$, the deformation space which admits a linear structure is the formal completion of the central leaf $C(x)$ through $x$, and therefore the Tate-linear formal subvarieties are smooth formal subvarieties of $C(x)^{/x}$. More precisely, if $\mathbb{Y}$ is the quasi-polarized $p$-divisible group of the abelian variety corresponding to $x$, the Tate-linear formal subvarieties are subvarieties of the form $V(\frak{b})\subseteq C^{/x}$ where $\frak{b}\subseteq\frak{gl}_{\mathbb{D}(\mathbb{Y})[p^{-1}]}$ is a Dieudonné-Lie subalgebra (see \ref{sezionetl} for the precise definition).

These Tate-linear formal varieties $V(\frak{b})$ play a central role in Chai's strategy for proving the Hecke orbit conjecture. The link is established by Chai--Oort's orbital rigidity theorem \cite[Theorem 4.1]{chaioortTL}, which states that the closure of a prime-to-$p$ Hecke orbit is Tate-linear at every point. Combining this with new results on the monodromy of the universal isocrystal $\mathcal{M}$ restricted to $V(\frak{b})$, D'Addezio--van Hoften were able to prove the Hecke orbit conjecture for the special fibre of all Shimura varieties of Hodge type with good reduction \cite{marcopol}. 

Aside from the results in these two papers, however, little has been written about Tate-linear subvarieties in the general case. In this work we prove the following strengthening of \cite[Theorem 6.1.1]{marcopol} and we deduce a rigidity property for subvarieties of $\calA_{g,k}$ which are Tate-linear at a point. In \textit{ibid.}, the authors show that $G(\mathcal{M}_{V(\frak{b})})\subseteq U(\frak{b})$, where $G(\mathcal{M}_{V(\frak{b})})$ denotes the monodromy group of $\mathcal{M}$ restricted to $V(\frak{b})$, and $U(\frak{b})$ is the connected subgroup of $\GL(\mathbb{D}(\mathbb{Y})[p^{-1}])$ with Lie algebra $\frak{b}$. 

\begin{thm}[Theorem \ref{mainthm}]\label{teomonodromia}
    Let $S\subseteq C(x)^{/x}$ be a smooth formal subvariety and suppose that $G(\mathcal{M}_{S})\subseteq U(\frak{b})$. Then $S\subseteq V(\frak{b})$.
\end{thm}

Since $G(\mathcal{M}_S)$ is always of the form $U(\frak{c})$ for some Dieudonné-Lie algebra $\frak{c}$, the following result is an immediate consequence of Theorem \ref{teomonodromia}. In the text we will give an independent proof.

\begin{prop}[Proposition \ref{propiso}]\label{coriso}
    The monodromy action gives an isomorphism $G(\mathcal{M}_{V(\frak{b})})\simeq U(\frak{b})$.
\end{prop}

More interestingly, by combining Theorem \ref{teomonodromia} with a result on local monodromy of \cite{marcopol} we get an affermative answer to a question of Chai \cite[Question 7.4]{chaioortTL} (see also \cite[Question 8.7.1]{marcopol}).

\begin{thm}[Theorem \ref{domandaC}]\label{rigiditaTL}
    Let $C$ be a central leaf of $\calA_g$ and let $Z\subseteq C$ be a smooth, connected, locally closed subvariety. If there exists $x\in Z(k)$ such that $Z^{/x}$ is Tate-linear, then $Z^{/y}$ is Tate-linear for all $y\in Z(k)$.
\end{thm}

Theorem \ref{rigiditaTL} has been proved by Chai in the ordinary case in \cite[5.3]{chaiordinary}. Chai has also proposed a characteristic $p$ analogue of Moonen's result linking subvarieties which are Tate-linear at a point and Shimura subvarieties \cite[Question 7.7]{chaioortTL}, which was later corrected in \cite[8.7.2]{marcopol}. It states that a smooth subvariety $Z\subseteq C(x)$, which is Tate-linear at a point, is in fact an irreducible component of a central leaf of the special fibre of a Hodge type Shimura subvariety of $\calA_g$. Chai \cite[7]{chaiordinary} and recently Jiang \cite{jiang} have linked this conjecture for \textit{ordinary} points to a Mumford--Tate type conjecture, which is equivalent to finding characteristic $0$ lifts of abelian varieties with a prescribed set of Hodge tensors. We hope that our work will serve as a stepping stone toward generalizing these results beyond the ordinary case, and toward establishing a theory of Tate-linear varieties in mixed characteristic.

\unit{An idea of the proof.} The proof of Theorem \ref{teomonodromia} relies on a refinement of the techniques developed in \cite[Section 6]{marcopol}, where the prismatization construction of Bhatt--Lurie and Drinfeld \cite{bhattlurie'} plays a central role. In order to motivate the use of the prismatization we outline in the following paragraph an analogy with monodromy groups of local systems in complex geometry (ignoring the several technical difficulties which are tackled in the paper). It gives an informal explanation of the proof of Proposition \ref{coriso} and the strategy behind the proof of Theorem \ref{teomonodromia}.

If $(X,x)$ is a pointed, smooth and connected complex variety, a local system $\mathcal{L}$ on $X$ corresponds to a representation of the fundamental group $\pi_1(X,x)$ via the following construction. The pullback $\widetilde{\mathcal{L}}$ of $\mathcal{L}$ to the universal cover $\widetilde{X}\to X$ is a $\pi_1(X)$-equivariant local system on $\widetilde{X}$. For topological reasons, it is ``split'' in the following sense: there exists a unique representation $\rho:\pi_1(X,x)\to\GL(\mathcal{L}_x)$ such that $\widetilde{\mathcal{L}}\simeq \mathcal{L}_x\times\widetilde{X}$ with diagonal action. The monodromy group of $\mathcal{L}$ is then the smallest algebraic subgroup of $\GL(\mathcal{L}_x)$ through which $\rho$ factors.

On the Tate-linear side, in Section \ref{acris} we construct a map $\widetilde{\Pi}(\frak{b})\to V(\frak{b})$ which is a torsor under a certain group $G_\frak{b}$. The space $\widetilde{\Pi}(\frak{b})$ is a perfect affine formal scheme, and should be thought in this analogy as a simply connected (quasisyntomic) cover of $V(\frak{b})$. In order to push the analogy further, we take the prismatization of these spaces \cite{bhattlurie'}. The map
\begin{equation*}
    q:\widetilde{\Pi}(\frak{b})^\Delta\to V(\frak{b})^\Delta
\end{equation*}
is a $G_\frak{b}^\Delta$-torsor, and the isocrystal $\mathcal{M}$ corresponds to a quasi-coherent sheaf on $V(\frak{b})^\Delta$, whose fibre over the natural $\W(k)$-point of $V(\frak{b})^\Delta$ is the Dieudonné module $\mathbb{D}(\mathbb{Y})[p^{-1}]$ of $\mathbb{Y}$. The group $G_\frak{b}$ has a natural representation $\rho_\frak{b}:G_\frak{b}\to\GL(\mathbb{D}(\mathbb{Y})[p^{-1}])$, and we show that
$q^*\mathcal{M}\simeq \mathbb{D}(\mathbb{Y}[p^{-1}])\times\widetilde{\Pi}(\frak{b})$ in a precise sense. It will then follow formally that the monodromy group of $\mathcal{M}_{V(\frak{b})}$ is the smallest algebraic subgroup of $\mathbb{D}(\mathbb{Y})$ through which $\rho_\frak{b}$ factors, and an efficient description of $\rho_\frak{b}$ relying on Scholze--Weinstein's results on Dieudonné theory for semiperfect rings shows that this group is equal to $U(\frak{b})$.

\unit{Organization of the paper.} The paper is structured as follows: Section \ref{acris} is a short summary of the theory of sustained $p$-divisible groups and Tate-linearity due to Chai--Oort (\cite{chaisustained},\cite{chaioortTL}), where we also use D'Addezio--van Hoften's formalism of Dieudonné-Lie algebras \cite{marcopol} in order to define Tate-linear subvarieties of $C(x)^{/x}$. In section \ref{dieudonnesw} we recall some results about Dieudonné theory for quasiregular semiperfect algebras from \cite{scholzeweinstein}, and we make some preliminary constructions. Namely, we describe the representation $\rho_\frak{b}$ of the previous pragraph. Finally, in section \ref{mainsec} we explain how to use the prismatization formalism to compute monodromy groups of isocrystals in our specific setting, and then we prove the main results.

\vspace{.5em}\noindent{\bf Notation and conventions.} (1) Throughout the paper $k$ is a field of characteristic $p>0$, and when $k$ is perfect we write $\W(k)$ for the Witt vectors of $k$ and $K=\W[1/p]$. We write $\calA_g$ for the $k$-scheme (or $k$-stack) parameterizing principally polarised abelian varieties with some fixed (implicit) level structure.

(2) We use uppercase roman letters $M,N,\dots$ for contravariant Dieudonné modules, lowercase gothic letters $\frak{a}^+,\frak{b}^+,\dots$ for covariant Dieudonné modules, and $\frak{a},\frak{b},\dots$ for covariant Dieudonné modules tensored with $\Qp$. We write $\mathbb{X},\mathbb{Y},\dots$ for $p$-divisible groups.

(3) If $\mathbb{X}$ is a $p$-divisible group over an $\Fp$-scheme $S$, we denote by $\mathbb{D}(\mathbb{X})$ its covariant Dieudonné crystal and by $\mathbb{D}^\vee(\mathbb{X})$ its contravariant Dieudonné crystal. Covariant Dieudonné crystals and modules have slopes contained in $[-1,0]$, and contravariant Dieudonné crystals and modules have slopes contained in $[0,1]$.

(4) If $R$ is an $\Fp$-algebra, we denote by $F$ its Frobenius, by $R^\flat$ the perfection $\varprojlim_FR$ and by $R_\text{perf}$ the coperfection $\varinjlim_FR$. We use the following non-standard terminology: an $f$-quasiregular semiperfect algebra is the quotient of a perfect ring by a regular sequence.

(5) If $X$ is a scheme and $x\in X$ is a point we write $X^{/x}$ for the formal completion of $X$ at $x$. 

(6) If $\calF$ and $\calG$ are sheaves on a site, we write $\sHom(\mathcal{F},\mathcal{G})$ for the hom-sheaf and $\Aut(\calF)$ for the sheaf of automorphisms of $\calF$.

\unit{Acknowledgements.} I sincerely thank Marco D'Addezio and Pol van Hoften for introducing me to their paper on the Hecke orbit conjecture, and for inviting me to think about these questions. This work has greatly benefited from discussions with them and from their continuous interest. I also thank Anna Cadoret and Thomas Agugliaro for interesting and useful discussions on the topics of this paper. The intellectual debt that this paper owes to the work Ching-Li Chai and Frans Oort is evident. 

The author received funding from the EU Horizon Europe research and innovation programme under the Marie Skłodowska-Curie grant agreement 101126554.

\section{Sustained $p$-divisible groups and Tate-linearity}\label{acris}

In this chapter we set the stage for studying monodromy groups of Tate-linear formal varieties. For a $k$-point $x\in\calA_g(k)$, the formal completion of the central leaf $C(x)^{/x}$ admits a \textit{linearization} which generalizes the Serre--Tate canonical coordinates of the ordinary case. This description was revealed by Chai--Oort's theory of sustained deformations of $p$-divisible groups, and it is the object of the following sections. After a brief discussion of the conventions on formal schemes and $p$-adic monodromy groups that we use throughout, we review the theory of sustained $p$-divisible groups and Tate-linearity following Chai--Oort and D'Addezio--van Hoften.

\subsection{Conventions on formal schemes}

We adopt a functor-of-points approach to formal schemes following \cite[Tag 0AHW]{stacksproject}, and the reader who looks for more detail is encouraged to read the discussion there. We will keep our discussion to the bare minimum as we will only use the basic properties of (countably indexed) affine formal schemes (over a field $k$). 

Recall that a thickening is a closed immersion of schemes which is also a homeomorphism - equivalently, a closed immersion whose sheaf of ideals has locally nilpotent sections.

\begin{deffa}
    An fppf sheaf $Z$ on $\Spec(k)$ is a (countably indexed) affine formal scheme if it can be written as a (countably indexed) colimit
    \begin{equation*}
        Z=\varinjlim_{i\in I}Z_i
    \end{equation*}
    where the $Z_i$ are affine schemes and the transition maps are thickenings.
\end{deffa}

In \cite[Tag 0AN1]{stacksproject} it is shown that countably indexed affine formal schemes are precisely the formal schemes of the form $\Spf(A)$, with $A$ a weakly admissible topological $k$-algebra with a countable fundamental system of neighbourhoods of $0$ (see \textit{ibid.}). From now on we assume all affine formal schemes to be countably indexed. If $B$ is a $k$-algebra and $Z=\Spf(A)$ we write 
\begin{equation*}
    Z(B):=Z(\Spec(B))=\varinjlim Z_i(\Spec(B)),
\end{equation*}
and we have $Z(B)=\Hom_{k,\text{cont}}(A,B)$ where $B$ is endowed with the discrete topology. We will write $Z^\alg=\Spec(A)$, which is equipped with a natural morphism of sheaves $Z\to Z^\alg$. 

Finally, we need will use the fact that affine formal schemes are sheaves for the fpqc topology over $\Spec(k)$ (\cite[Tag 0AQD]{stacksproject}).

\begin{ex}\label{deftpuc}
    Let $\mathbb{X}$ be a connected $p$-divisible group over $\Spec(k)$. As explained in \cite[3]{scholzeweinstein}, we can associate to it three affine formal group schemes.
    \begin{enumerate}
        \item The sheaf $\mathbb{X}$ itself is a smooth formal scheme over $k$ of the form $\Spf(k[\![x_1,\dots,x_n]\!])$.
        \item Its Tate module $T_p\mathbb{X}=\varprojlim_p\mathbb{X}[p^n]$ is representable by an $f$-semiregular quasiperfect affine group scheme:
        \begin{equation*}
            T_p\mathbb{X}\simeq\Spec\big(k[x_1^{\pmeninf},\dots,x_n^{\pmeninf}]/(x_1,\dots,x_n)\big).
        \end{equation*} 
        \item Its universal cover $\widetilde{\mathbb{X}}=T_p\mathbb{X}[1/p]=\varinjlim_pT_p\mathbb{X}$ is of the form $\Spf(k[\![x_1^{\pmeninf},\dots,x_n^{\pmeninf}]\!])$. To see this last assertion, note that since $\mathbb{X}$ is connected we have $\varinjlim_pT_p\mathbb{X}=\varinjlim_FT_p\mathbb{X}$, see \cite[Proof of Proposition 3.1.3]{scholzeweinstein}.
    \end{enumerate} 
    By definition, there is a short exact sequence 
    \begin{equation*}
        0\to T_p\mathbb{X}\to\widetilde{\mathbb{X}}\to\mathbb{X}\to0
    \end{equation*}
    of fpqc sheaves on $\textnormal{Sch}_k$.
\end{ex} 

\unit{Torsors on affine formal schemes.} Let $G$ be an affine group scheme over $k$ and let $Z=\varinjlim Z_i$ be an affine formal scheme. A $G$-torsor over $Z$ is the datum of fpqc $G$-torsors $T_i\to Z_i$ for all $i$, and maps $T_i\to T_j$ over $Z_i\to Z_j$ such that the square
\[\begin{tikzcd}
	{T_i} & {T_j} \\
	{Z_i} & {Z_j}
	\arrow[from=1-1, to=1-2]
	\arrow[from=1-1, to=2-1]
	\arrow[from=1-2, to=2-2]
	\arrow[from=2-1, to=2-2]
\end{tikzcd}\]
is cartesian. It is then clear that $T=\varinjlim T_i$ is an affine formal scheme over $k$ and that $T\to Z$ is a $G$-torsor of fpqc sheaves. Conversely, any $G$-torsor $T\to Z$ of fpqc sheaves gives rise to a $G$-torsor in the former sense, so this definition does not depend on the presentation $Z=\varinjlim Z_i$.

\subsection{Monodromy groups of isocrystals}\label{monodromyiso}

\unit{The category of isocrystals.} Let $X$ be an $\Fp$-scheme. We denote by $\cris(X)$ the category of coherent crystals over $X$, which is a $\Hcris{0}(X)$-linear category. The category of isocrystals of $X$ is the isogeny category of $\cris(X)$, and is denoted $\Isoc(X)$. 

\unit{Tannakian structure.} Let $X$ be a Noetherian Frobenius-smooth $\Fp$-scheme \cite[Lemma 2.1.1]{drinfeldstacky}. For example $X$ could be smooth over a field with a finite $p$-basis, or $X$ could be of the form $\Spec(k[\![x_1,\dots,x_n]\!])$ with $k$ a perfect field. If $X$ is irreducible, the perfect ring
\begin{equation*}
    \kappa=\bigcap_n\Gamma(X,\mathcal{O}_X)^{p^n}
\end{equation*}
is a field, and $\Hcris{0}(X)=W(\kappa)$ \cite[Proposition 2.8.3]{drinfeldstacky}. If we write $K=W(\kappa)[1/p]$, by a result of Drinfeld \cite[Corollary 3.3.3]{drinfeldstacky} the category $\Isoc(X)$ is a $K$-linear Tannakian category. The choice of a point $\eta:\Spec(\kappa)\to X$ defines a fibre functor $\Isoc(X)\to\textnormal{Vec}_{K}$, which allows us to define the monodromy group $G(\mathcal{M},\eta)$ of an isocrystal $\mathcal{M}$ as the Tannaka group of the Tannakian subcategory $\langle\mathcal{M}\rangle^\otimes\subseteq\Isoc(X)$. Thus the group $G(\mathcal{M},\eta)$ is a linear algebraic subgroup of $\GL(\mathcal{M}_\eta)$. 

When $X\simeq\Spec(k[\![x_1,\dots,x_n]\!])$ for a perfect field $k$ we have $\kappa=k$, and since there is only one choice for $\eta$ we abbreviate the monodromy group to $G(\mathcal{M})$. If $\mathcal{M}$ is an isoclinic $F$-isocrystal on $X$ it admits the slope filtration \cite[3.2.6]{marcopol}, and then \cite[A.1.2]{marcopol} implies that $G(\mathcal{M})$ is a unipotent algebraic group.

\unit{Monodromy for formal schemes.} There are two issues which prevent us from working directly with formal schemes in this paper: as of yet, there is no Tannakian formalism for isocrystals on a formal scheme, and no theory of prismatization. In order to define the monodromy group of (the isocrystal associated to) a $p$-divisible group over a formal scheme, we resort to an algebrization result due to Messing.

Let $A=k[\![x_1,\dots,x_n]\!]$ where $k$ is a perfect field, let $Z=\Spf(A)$ and $Z^\alg=\Spec(A)$. If we define a $p$-divisible group on $Z=\varinjlim Z_i$ as a compatible system of $p$-divisible groups on the $Z_i$, by \cite[II.4.16]{messingdeformation} the natural functor
\begin{equation*}
    \Big\{p\text{-divisible groups over }Z^\alg\Big\}\to \Big\{p\text{-divisible groups over }Z\Big\}
\end{equation*}
is an equivalence of categories. Thus, if $\mathbb{X}$ is a $p$-divisible group over $Z$ and $\mathcal{M}$ is the corresponding $F$-isocrystal, we define
\begin{equation*}
    G(\mathcal{M}):=G(\mathcal{M}^\alg),
\end{equation*}
where $\mathcal{M}^\alg$ denotes the $F$-isocrystal associated to the algebraization $\mathbb{X}^\alg$ of the $p$-divisible group $\mathbb{X}$. 

\subsection{Sustained $p$-divisible groups}

\unit{Definition and first properties.} Sustained deformations of $p$-divisible groups were introduced by Chai and Oort to give a scheme-theoretic description of the central leaves of $\calA_g$. We recall the definition and some results from the original reference \cite{chaisustained}, but we assume from the start that all schemes live over a fixed field $k$. Let $\mathbb{Y}$ be a 
$p$-divisible group over $k$.

\begin{deffa}[{\cite[2.1]{chaisustained}}]
    Let $S$ be a $k$-scheme. A 
    $p$-divisible group $\mathbb{X}$ over $S$ is sustained modeled on $\mathbb{Y}$ if, for all $n\ge0$, the scheme
    \begin{equation*}
        T_{\mathbb{X},n}:=\Isom_S(\mathbb{X}[p^n],\mathbb{Y}[p^n]_S)\to S
    \end{equation*}
    is an $\Aut(\mathbb{Y}[p^n])$-torsor over $S$. 
\end{deffa}

In other words, a sustained $p$-divisible group modeled on $\mathbb{Y}$ is a $p$-divisible group $\mathbb{X}$ such that $\mathbb{X}[p^n]$ is isomorphic to $\mathbb{Y}[p^n]$ over some fppf cover of $S$. It will be useful for our purposes to take this description to the infinite level: if $\mathbb{X}$ is sustained modeled on $\mathbb{Y}$, the maps of torsors $T_{\mathbb{X},n+1}\to T_{\mathbb{X},n}$ are compatible with the morphisms of structure groups $\Aut(\mathbb{Y}[p^{n+1}])\to\Aut(\mathbb{Y}[p^n])$. Thus, taking the limit $T_\mathbb{X}=\varprojlim_n T_{\mathbb{X},n}$ we get the affine $S$-scheme
\begin{equation*}
    T_\mathbb{X}=\Isom_S(\mathbb{X},\mathbb{Y}_S)\to S
\end{equation*}
which is an $\Aut(\mathbb{Y})=\varprojlim_n\Aut(\mathbb{Y}[p^n])$-torsor in the fpqc topology. For a $k$-algebra $R$, the set $T_\mathbb{X}(R)$ is the set of isomorphisms of $p$-divisible groups $\mathbb{X}_R\simeq \mathbb{Y}_R$.

Conversely, if $T\to S$ is an $\Aut(\mathbb{Y})$-torsor for the fpqc topology, the assignment 
\begin{equation}\label{xn}
    X_n=T\times^{\Aut(\mathbb{Y})}\mathbb{Y}[p^n]
\end{equation} 
gives rise to a family of finite flat group schemes over $S$. By \cite[3.6]{chaisustained} the sheaf $\mathbb{X}=\varinjlim_n X_n$ is a $p$-divisible group over $S$ which is sustained modeled on $\mathbb{Y}$ and satisfies $\mathbb{X}[p^n]=X_n$. Furthermore, the assignments $\mathbb{X}\mapsto T_\mathbb{X}$ and $T\mapsto \mathbb{X}$ are inverses of one another, so it is equivalent to give a $p$-divisible group sustained on $\mathbb{Y}$, or the corresponding $\Aut(\mathbb{Y})$-torsor.

These considerations extend to formal schemes: if $Z=\varinjlim Z_i$ is an affine formal scheme and $\mathbb{X}$ is a $p$-divisible group on $Z$ which is sustained modeled on $\mathbb{Y}$, i.e. the pullback of $\mathbb{X}$ to $Z_i$ is sustained modeled on $\mathbb{Y}$ for all $i$, then the above discussion produces $\Aut(\mathbb{Y})$-torsors $T_i\to Z_i$ which assemble to a torsor $T\to Z$. The description of $T$ as a sheaf on $k$-algebras is the same as for ordinary schemes.

\unit{The groups $\Aut(\mathbb{Y})$ and $\Aut(\widetilde{\mathbb{Y}})$.} In light of the equivalence between sustained $p$-divisible groups and torsors, we now spend some words on the structure of the group scheme $\Aut(\mathbb{Y})$. Our starting point is Chai--Oort's construction of the internal Hom $p$-divisible group $\mathcal{H}_\mathbb{Y}$, see \cite{chaisustained} and \cite[4.1]{marcopol}. 

Roughly speaking, for $n\ge0$ the group $\mathcal{H}_\mathbb{Y}[p^n]$ is defined as the image of the map of fpqc sheaves $\End(\mathbb{Y})\to\End(\mathbb{Y}[p^n])$: it parametrizes endomorphisms of $\mathbb{Y}[p^n]$ which extend to full endomorphism of $\mathbb{Y}$ over an fpqc cover of the base. The notation is justified by the fact that the direct limit $\mathcal{H}_\mathbb{Y}:=\varinjlim_n\mathcal{H}_\mathbb{Y}[p^n]$ is a $p$-divisible group \cite[3.2]{chaisustained}, and by construction $T_p\mathcal{H}_\mathbb{Y}=\End(\mathbb{Y})$. 

Note that each $\mathcal{H}_\mathbb{Y}[p^n]$ is naturally a ring scheme, and by definition we have 
\begin{equation*}
    \Aut(\mathbb{Y})=(T_p\mathcal{H}_\mathbb{Y})^\times=\varprojlim_n(\mathcal{H}_\mathbb{Y}[p^n])^\times,
\end{equation*}
so $\Aut(\mathbb{Y})$ is a profinite group scheme over $k$. In the rest of the paper we will focus on the connected components of the group schemes $T_p\mathcal{H}_{\mathbb{Y}}$ and $\Aut(\mathbb{Y})$. For a $k$-scheme $Z$, the connected group scheme $T_p\mathcal{H}_\mathbb{Y}^\circ=\End^\circ(\mathbb{Y})$ parameterizes endomorphisms of $\mathbb{Y}_Z$ which are $0$ on $Z_\textnormal{red}$, and $\Aut^\circ(\mathbb{Y})$ parameterizes those automorphisms which are the identity on $Z_\textnormal{red}$. 

For our purposes, it is appropriate to think of $T_p\mathcal{H}_\mathbb{Y}^\circ$ as a sort of Lie algebra of the group $\Aut^\circ(\mathbb{Y})$, just like $\textnormal{End}(V)$ is the Lie algebra of $\GL(V)$ for a vector space $V$. Indeed, just like for $\textnormal{End}(V)$, the ring structure on $T_p\mathcal{H}_{\mathbb{Y}}^\circ$ endows it with a natural Lie bracket. This idea is due to Chai--Oort and developed extensively in \cite{marcopol}. To make this precise, we introduce the functor $\Aut(\widetilde{\mathbb{Y}})$ parameterizing quasi-isogenies of $\mathbb{Y}$, which is the group of units of the universal cover $\widetilde{\mathcal{H}}_{\mathbb{Y}}$ of $\mathcal{H}_{\mathbb{Y}}$, and its connected component $\Aut^\circ(\widetilde{\mathbb{Y}})$. By the discussion of Example \ref{deftpuc} above we have $\Aut^\circ(\widetilde{\mathbb{Y}})=\varinjlim_F\Aut^\circ(\mathbb{Y})$, so $\Aut^\circ(\widetilde{\mathbb{Y}})$ is an affine formal group scheme over $k$.

The point is that there is a well-defined exponential map
\begin{equation}\label{exp}
    \exp:\widetilde{\mathcal{H}}_\mathbb{Y}^\circ\to\Aut^\circ(\mathbb{\widetilde{Y}}),\quad f\mapsto \sum_{n\ge0}f^n/n!
\end{equation}
which is an isomorphism of affine formal schemes \cite[4.1.9]{marcopol}, whose inverse is a similarly-defined logarithm map $\log:\Aut^\circ(\mathbb{\widetilde{Y}})\to\widetilde{\mathcal{H}}_\mathbb{Y}^\circ$. The group law on $\Aut^\circ(\mathbb{\widetilde{Y}})$ can be described uniquely in terms of $\widetilde{\mathcal{H}}_\mathbb{Y}^\circ$ and its Lie bracket. Indeed, for purely formal reasons we have the equality
\begin{equation*}
    \exp(f)\circ\exp(g)=\exp(\textnormal{BCH}(f,g)),
\end{equation*}
where $\textnormal{BCH}(x,y)$ is the Baker-Campbell-Hausdorff power series - see \cite[4.3]{marcopol} for details.

One technical issue with the isomorphism \eqref{exp} is that it does not necessarily restrict to a map from $T_p\mathcal{H}_\mathbb{Y}^\circ$ to $\Aut^\circ(\mathbb{Y})$. However, since $p^{n}/n!\in\Zp$ for all $n\ge0$, we have in fact $\exp(f)\in\Aut^\circ(\mathbb{Y})$ for $f\in pT_p\mathcal{H}_\mathbb{Y}^\circ$. Writing the short exact sequence
\begin{equation}\label{exppiquadro}
    0\to pT_p\mathcal{H}_\mathbb{Y}^\circ\xto{}{\exp}\Aut^\circ(\mathbb{Y})\to H\to 0,
\end{equation}
it is possible to show that $H$ is an inifinitesimal group scheme over $k$, leveraging the fact that the logarithm also converges (see \cite[5.5.3]{marcopol}).

An important observation is the fact that $p$-divisible groups  endowed with Lie brackets correspond, via Dieudonné theory, to Dieudonné modules endowed with compatible Lie brackets. This is framed as compatibility of Dieudonné theory with bilinear maps in \cite[4.2.12]{marcopol} and is essential in the next subsection. The upshot is that the pair $(\widetilde{\mathcal{H}}^\circ_\mathbb{Y},[-,-])$, and thus the group scheme $\Aut^\circ(\widetilde{\mathbb{Y}})$, can be recovered from just the Dieudonné module of $\mathcal{H}_\mathbb{Y}^\circ$ and its corresponding Lie algebra. The following description of this Dieudonné module is due to Chai--Oort.

\begin{prop}\label{dieudonneinternalhom}
    Suppose that $k$ is a perfect field, and let $\frak{a}_\mathbb{Y}^+$ be the covariant Dieudonné module of $\mathcal{H}_\mathbb{Y}^\circ$.  By Dieudonné theory, $\frak{a}_\mathbb{Y}^+$ is naturally a subset of $\textnormal{End}_{\W(k)}(\mathbb{D}(\mathbb{Y}))$, which is a $\W(k)[F]$-module with slopes in $[-1,1]$. It is identified with the largest $\W(k)$-submodule with slopes in $[-1,0[$.
\end{prop}
\noindent{\it Proof.}
See \cite[3.2]{chaisustained}.
$\qed$

\unit{Infinitesimal sustained deformations.} Now let us introduce the functor parameterizing infinitesimal sustained deformations of $\mathbb{Y}$. We denote by $\textbf{Def}(\mathbb{Y})$ the functor of global deformations of $\mathbb{Y}$, as discussed in \cite[5.4.6]{marcopol}: its value on a $k$-algebra $R$ is the set of isomorphism classes of pairs $(\mathbb{X},\rho)$, where $\mathbb{X}$ is a $p$-divisible group over $R$ and $\rho:\mathbb{Y}_{R}\to \mathbb{X}$ is a quasi-isogeny which restricts to an isomorphism on $R_\text{red}$. By rigidity of morphisms up to isogeny, $\textbf{Def}(\mathbb{Y})$ corresponds to the usual deformation functor when restricted to the category of Artin $k$-algebras.

\begin{deffa}
    The functor of infinitesimal sustained deformations of $\mathbb{Y}$ is defined by
    \begin{equation*}
        \defsus(\mathbb{Y})(R)=\{(\mathbb{X},\rho)\in\textbf{Def}(\mathbb{Y})(R)\text{ s.t. }\mathbb{X}\text{ is sustained modeled on }\mathbb{Y}\}
    \end{equation*}
    for a $k$-algebra $R$.
\end{deffa}
By \cite[4.3]{chaisustained} the functor $\defsus(\mathbb{Y})$ is representable by a smooth affine formal scheme of the form $\Spf(k[\![x_1,\dots,x_n]\!])$, and it comes with a universal sustained 
$p$-divisible group $\mathcal{Y}\to\defsus(\mathbb{Y})$ (see also \cite[5.4.6]{marcopol}). If $\mathcal{T}$ denotes the corresponding torsor, then for any $k$-algebra $R$ we have
\begin{equation*}
    \mathcal{T}(R)=\{((\mathbb{X},\rho),\varphi)\text{ with }(\mathbb{X},\rho)\in\defsus(\mathbb{Y})(R)\text{ and }\varphi:\mathbb{X}\to \mathbb{Y}_R\text{ an isomorphism}\},
\end{equation*}
and the natural map $\mathcal{T}(R)\to\defsus(\mathbb{Y})(R)$ sends $((\mathbb{X},f),\varphi)$ to $(\mathbb{X},f)$. From the affine formal scheme $\mathcal{T}$ we can extract an $\Aut^\circ(\mathbb{Y})$-torsor over $\defsus(\mathbb{Y})$, defined as the subfunctor $\mathcal{T}^\circ\subseteq\mathcal{T}$ consisting of the triples $((X,\rho),\varphi)$ such that
\begin{equation*}
    Y_{R_\text{red}}\xto{}{\rho}X_{R_\text{red}}\xto{}{\varphi}Y_{R_\text{red}}
\end{equation*}
is the identity. It is easy to see that the map
\begin{equation*}
    \Aut^\circ(\widetilde{\mathbb{Y}})\to \mathcal{T}^\circ,\quad \rho\mapsto((\mathbb{Y}_R,\rho),\mathrm{id}_{\mathbb{Y}_R})
\end{equation*}
is an isomorphism of sheaves and is $\Aut^\circ(\mathbb{Y})$-equivariant.

Therefore we can write $\defsus(\mathbb{Y})$ as a quotient of fpqc sheaves\footnote{To avoid the set-theoretic difficulties that fpqc sheafification usually poses, one can take the sheafification for any coarser topology containing limits of countable projective systems of finite flat covers.}
\begin{equation}\label{defsusquoziente}
    \defsus(\mathbb{Y})\simeq\Aut^\circ(\widetilde{\mathbb{Y}})\big/\Aut^\circ(\mathbb{Y}).
\end{equation}
We invite the reader to think of $\Aut^\circ(\widetilde{\mathbb{Y}})$ as a simply connected cover of the formal scheme $\defsus(\mathbb{Y})$.

\unit{Relationship with central leaves.} In \cite{chaisustained} the authors also consider the case of a quasi-polarized $p$-divisible group $(\mathbb{Y},\lambda)$ over $k$. They define the notion of a sustained quasi-polarized $p$-divisible group modeled on $(\mathbb{Y},\lambda)$, and they show the equivalence with $\Aut(\mathbb{Y},\lambda)$-torsors. They introduce the deformation functor $\defsus(\mathbb{Y},\lambda)$ and show that it is a smooth formal subscheme of $\defsus(\mathbb{Y})$. With these results in hand, we can then state one of the main results of \cite{chaisustained}, which strongly motivates the definitions and constructions of the previous paragraphs. 

Let $k=\Fpbar$, fix an integer $g\ge1$ and a point $x\in\mathcal{A}_g(k)$ which corresponds to a principally polarized abelian variety $(A_x,\lambda_x)$ of dimension $g$. The central leaf of $x$ as defined by Oort in \cite{oortfoliations} is the locally closed reduced subspace $C(x)\subseteq\calA_g$ which consists of the points $y\in\calA_g(k)$ such that $(A_x,\lambda_x)[p^\infty]\simeq(A_y,\lambda_y)[p^\infty]$. The following theorem of Chai--Oort describes the whole functor of points of $C(x)$. It thus gives a description of the formal completion $C(x)^{/x}$, which is not easily derived from  Oort's pointwise definition of $C(x)$.

\begin{thm}\label{centralsusglobale}
    The closed subscheme $C(x)\subseteq\calA_g$ parameterizes the families of principally polarized abelian varieties whose quasi-polarized $p$-divisible group is sustained modeled on $(A_x,\lambda_x)[p^\infty]$.
\end{thm}
\noindent{\it Proof.}
See \cite[2.8]{chaisustained}.
$\qed$

\begin{cor}
\label{centralsus}
    There is an identification
    \begin{equation*}
        C(x)^{/x}\simeq\defsus(A_x[p^\infty],\lambda_x[p^\infty]),
    \end{equation*}
    compatibly with Serre-Tate's isomorphism.
\end{cor}
\noindent{\it Proof.}
Combine Theorem \ref{centralsusglobale} with the Serre--Tate isomorphism $\calA_g^{/x}\simeq\textbf{Def}(A_x[p^\infty],\lambda_x[p^\infty])$.
$\qed$

\subsection{Tate-linear formal subvarieties}\label{sezionetl}

To describe the condition of Tate-linearity we use an approach which is halfway between that of \cite{marcopol} and Chai-Oort's approach in \cite{chaioortTL}. The reader can find in section $7$ of the former a dictionary between the two points of view. The reader looking for more details (\textit{e.g.} the canonical filtration on Tate-linear varieties) and motivation for the following definitions is encouraged to consult those two papers. 

We assume that the base field $k$ is perfect. Combining the quasi-polarized analogue of \eqref{defsusquoziente} with Corollary \ref{centralsus}, we can write the central leaf through $x\in\calA_g(k)$ as a quotient of fpqc sheaves
\begin{equation}\label{centralquot}
    C(x)^{/x}\simeq\Aut^\circ(\widetilde{\mathbb{Y}},\lambda)/\Aut^\circ(\mathbb{Y},\lambda)\subseteq\Aut^\circ(\widetilde{\mathbb{Y}})/\Aut^\circ(\mathbb{Y})\simeq\defsus(\mathbb{Y}),
\end{equation}
where $(\mathbb{Y},\lambda)=(A_x,\lambda_x[p^\infty])$. As discussed above, the formal group $\Aut^\circ(\widetilde{\mathbb{Y}})$ is uniquely determined by a datum of (semi-)linear algebra: the Dieudonné module $\frak{a}_\mathbb{Y}^+$ together with its natural Lie bracket. In this section we show that $C(x)^{/x}$ is a Tate-linear formal subvariety of $\defsus(\mathbb{Y})$, \textit{i.e.} a formal subvariety arising from subobjects of this linear algebra datum.

\unit{Dieudonné-Lie algebras.} A Dieudonné-Lie algebra (over $K$) is a covariant Dieudonné module $\frak{m}$ over $K$, endowed with a Lie bracket which is compatible with the Frobenius-linear endomorphism of $\frak{m}$, see \cite[4.2.1]{marcopol}. The key examples for us are $\frak{a}_\mathbb{Y}:=\frak{a}_\mathbb{Y}^+[p^{-1}]$, which is nilpotent as a Lie algebra \cite[4.1.4]{marcopol}, and its Dieudonné-Lie subalgebras $\frak{b}\subseteq\frak{a}_\mathbb{Y}$. We will also write $\frak{b}^+=\frak{b}\cap\frak{a}_\mathbb{Y}^+$ for the natural integral lattice of $\frak{b}$.

A Dieudonné-Lie algebra $\frak{b}\subseteq\frak{a}_\mathbb{Y}$ corresponds, by Dieudonné theory, to a $p$-divisible subgroup $\mathcal{H}^\circ_{\mathbb{Y},\frak{b}}$ of $\mathcal{H}^\circ_{\mathbb{Y}}$. By bilinear Dieudonné theory \cite[4.3.1]{marcopol}, the sheaf
\begin{equation*}
    \Aut^\circ_\frak{b}(\widetilde{\mathbb{Y}}):=\exp\big(\widetilde{\mathcal{H}}^\circ_{\mathbb{Y},\frak{b}}\big)
\end{equation*}
is a closed formal subgroup of $\Aut^\circ(\widetilde{\mathbb{Y}})$, denoted $\widetilde{\Pi}(\frak{b})$ in \cite{marcopol}. We also define
\begin{equation*}
    \Aut^\circ_\frak{b}(\mathbb{Y}):= \Aut^\circ_\frak{b}(\widetilde{\mathbb{Y}})\cap\Aut^\circ(\mathbb{Y}).
\end{equation*}

\begin{rem}\label{hodgesub}
    As explained in \cite[5.5.8]{marcopol}, these groups appear naturally when generalizing the quotient presentation \eqref{centralquot} to central leaves of Shimura subvarieties of $\calA_g$ with good reduction. Let $(G,X)\subseteq(\mathrm{GSp}_{2g},\pm\mathbb{H}_g)$ be a Shimura subdatum with good reduction at $p$, whose special fibre $\overline{\mathcal{S}}$ embeds into $\calA_g$. The universal $p$-divisible group on $\overline{\mathcal{S}}$ carries some $F$-invariant crystalline tensors $(\xi_i)$ defined by the embedding $G\subseteq\mathrm{GSp}_{2g}$ \cite[5.2.1]{marcopol}. Then $\frak{b}$ is the subalgebra of $\frak{a}_\mathbb{Y}$ acting trivally on these tensors, and the group $\Aut^\circ_\frak{b}(\mathbb{Y})$ parameterizes the automorphisms of $\mathbb{Y}$ fixing the tensors $\xi_i|_x\in\mathbb{D}(\mathbb{Y})^\otimes$.
\end{rem}

It follows from \eqref{exppiquadro} that there is a short exact sequence of affine group schemes
\begin{equation*}
    0\to pT_p\mathcal{H}^\circ_{\mathbb{Y},\frak{b}}\xto{}{\exp}\Aut^\circ_\frak{b}(\mathbb{Y})\to H'\to0
\end{equation*}
with $H'$ is an infinitesimal group scheme.  We will use the following technical result later on, but we postpone the proof to the end of the chapter.

\begin{lem}\label{autbeqrsp}
    Writing $\Aut^\circ_\frak{b}(\mathbb{Y})=\Spec(C_\frak{b})$, the ring $C_\frak{b}$ is $f$-quasiregular semiperfect. More precisely, it is of the form $k[x_1^{\pmeninf},\dots,x_m^{\pmeninf}]/(x_1,\dots,x_m)$.
\end{lem}

\unit{Definition of Tate-linearity.} Let $\frak{b}\subseteq\frak{a}_Y$ be a Dieudonné-Lie $K$-subalgebra. Bby \cite[Proposition 3.1]{chaioortTL}, the quotient sheaf
\begin{equation*}
    V(\frak{b}):=\Aut^\circ_\frak{b}(\widetilde{\mathbb{Y}})/\Aut^\circ_\frak{b}(\mathbb{Y})\subseteq\Aut^\circ(\widetilde{\mathbb{Y}})/\Aut(\mathbb{Y})=\defsus(\mathbb{Y})
\end{equation*}
is a closed formal subscheme of $\defsus(\mathbb{Y})$ of the form $\Spf(k[\![y_1,\dots,y_m]\!])$.

\begin{rem}
    Consider $x\in\calA_g(k)$ with corresponding quasi-polarized $p$-divisible group $(\mathbb{Y},\lambda)=(A_x[p^\infty],\lambda_x[p^\infty])$. If $\frak{b}$ denotes the intersection $\frak{a}_\mathbb{Y}\cap\mathfrak{gsp_{(\mathbb{D}(\mathbb{Y})[p^{-1}],\lambda)}}$ inside $\mathfrak{gl}_{\mathbb{D}(\mathbb{Y})[p^{-1}]}$, then we have $C(x)^{/x}=V(\frak{b})$ as formal subspaces of $\defsus(\mathbb{Y})$ \cite[5.5.8]{marcopol}. In the context of Remark \ref{hodgesub}, the Dieudonné-Lie subalgebra $\frak{b}\subseteq\frak{a}_\mathbb{Y}$ is the subalgebra acting trivially on the crystalline tensor defined by the quasi-polarization of $\mathbb{Y}$.
\end{rem}

\begin{deffa}
    (1) A formal subvariety $Z\subseteq\defsus(\mathbb{Y})$ is Tate-linear if it is of the form $V(\frak{b})$ for some Dieudonné-Lie $K$-subalgebra $\frak{b}$. 

    (2) Let $C\subseteq\calA_g$ be a central leaf. A closed subvariety $Z\subseteq C$ is Tate-linear at a point $x\in Z(k)$ if $Z^{/x}\subseteq C^{/x}$ is Tate-linear under the identification of $C^{/x}$ as a subspace of $\defsus(A_x[p^\infty])$ of Corollary \ref{centralsus}.
\end{deffa}

\unit{Unipotent groups associated to Dieudonné-Lie algebras.} A Dieudonné-Lie algebra $\frak{b}\subseteq\frak{a}_\mathbb{Y}$ is in particular a nilpotent Lie algebra over $K$. As such, it defines a smooth connected unipotent group scheme $U(\frak{b})$ over $K$ by integration of the Lie bracket. Via the inclusion $\frak{a}_\mathbb{Y}\subseteq\textnormal{End}_K(\mathbb{D}(\mathbb{Y})[p^{-1}])$, the group $U(\frak{b})$ is naturally identified with the unique closed connected subgroup of $\GL(\mathbb{D}(\mathbb{Y})[p^{-1}])$ with Lie algebra $\frak{b}$. In some respects it will be natural to work over $\W(k)$, so we also introduce $U(\frak{b})_{\textnormal{int}}$, the schematic closure of $U(\frak{b})$ in $\GL(\mathbb{D}(\mathbb{Y}))$, which is a flat group scheme over $\W(k)$. 

A brief description of the Hopf algebras underlying these two group schemes is in order. Let us write
\begin{equation}\label{anellideigruppi}
    \GL(\mathbb{D}(\mathbb{Y}))=\Spec(R^+_\mathbb{Y}),\quad U(\frak{b})=\Spec(R_\frak{b}),
\end{equation} 
so $R_\frak{b}$ is a quotient of $R_\mathbb{Y}=R_\mathbb{Y}^+[p^{-1}]$. Then we have an isomorphism
\begin{equation}\label{anelloub}
    U(\frak{b})\simeq\Spec(\textnormal{Sym}_{K}(\frak{b}^\vee)),
\end{equation}
given by matrix exponentiation. On the other hand, $U(\frak{b})_{\textnormal{int}}=\Spec(R_\frak{b}^+)$ where $R_\frak{b}^+$ is the image of the composition $R_\mathbb{Y}^+\hookrightarrow R_\mathbb{Y}\onto{}R_\frak{b}$.

\unit{Our setup for studying monodromy groups.} The following setting sums up the configuration of objects and maps that we consider throughout the paper.

\begin{setting}\label{setting}
    Let $x\in\mathcal{A}_g(k)$ and let $\mathbb{Y}=(A_x[p^\infty],\lambda_x[p^\infty])$. If $\frak{b}\subseteq\frak{a}_Y$ is a Dieudonné-Lie subalgebra we have a commutative diagram
    \begin{equation}\label{diagramsetting}
    \begin{tikzcd}
	{\Aut_\frak{b}^\circ(\mathbb{Y})} & {\Aut^\circ(\mathbb{Y})} \\
	{\Aut^\circ_\frak{b}(\widetilde{\mathbb{Y}})} & {\Aut^\circ(\widetilde{\mathbb{Y}})} \\
	{V(\frak{b})} & {\defsus(\mathbb{Y})}
	\arrow[from=1-1, to=1-2]
	\arrow[from=1-1, to=2-1]
	\arrow[from=1-2, to=2-2]
	\arrow[from=2-1, to=2-2]
	\arrow[from=2-1, to=3-1]
	\arrow[from=2-2, to=3-2]
	\arrow[hook, from=3-1, to=3-2]
    \end{tikzcd}
    \end{equation}
    where each bottom vertical map is a torsor for the group above it.
    
    Let $\mathcal{M}$ denote the isocrystal associated to  the universal $p$-divisible group over $\defsus(Y)$, and write $\mathcal{M}_{V(\frak{b})}$ for its pullback to $V(\frak{b})$. The monodromy group $G(\mathcal{M}_{V(\frak{b})})$ is canonically a subgroup of $\GL(\mathbb{D}(\mathbb{Y})[p^{-1}])$, and since $\mathcal{M}$ comes from an $F$-isocrystal its Lie algebra is in fact contained in $\mathbb{D}(\mathcal{H}_\mathbb{Y})$ \cite[5.5.1]{marcopol}. Proposition \ref{coriso}, which says that $G(\mathcal{M}_{V(\frak{b})})=U(\frak{b})$, is equivalent to the fact that $\Lie(G(\mathcal{M}_{V(\frak{b}}))$ is contained in $\frak{b}$. And Theorem \ref{teomonodromia}, which says that $V(\frak{b})$ is the largest smooth formal subscheme $S\subseteq\defsus(Y)$ such that $G(\mathcal{M}_{S})\subseteq U(\frak{b})$, can also be restated in terms of Lie algebras.

    Finally, as explained in \cite[6.3.3]{marcopol} the torsor $\Aut^\circ_\frak{b}(\widetilde{\mathbb{Y}})\to V(\frak{b})$ can be algebraized to an $\Aut_\frak{b}^\circ(\mathbb{Y})$-torsor $\Aut^\circ_\frak{b}(\widetilde{\mathbb{Y}})^\alg\to V(\frak{b})^\alg$. As a scheme, $\Aut^\circ_\frak{b}(\widetilde{\mathbb{Y}})^\alg$ is isomorphic to the perfection of $V(\frak{b})^\alg$. One way of seeing this is the following: taking the limit of the morphism $\Aut^\circ_\frak{b}(\widetilde{\mathbb{Y}})^\alg\to V(\frak{b})^\alg$ along the Frobenius maps gives rise to a $\varprojlim_F\Aut^\circ_\frak{b}(\mathbb{Y})$-torsor $\varprojlim_F\Aut^\circ_\frak{b}(\widetilde{\mathbb{Y}})^\alg\to\varprojlim_FV(\frak{b})^\alg$. However $\varprojlim_F\Aut^\circ_\frak{b}(\mathbb{Y})$ is the trivial group scheme by Lemma \ref{autbeqrsp}, and $\Aut^\circ_\frak{b}(\widetilde{\mathbb{Y}})^\alg$ is perfect (see also \cite[4.3.13]{marcopol}). 
    
    By the discussion at the end of Subsection \ref{monodromyiso}, our goal is therefore to understand the monodromy groups $G(\mathcal{M}_{V(\frak{b})^\alg})$, where $\mathcal{M}_{V(\frak{b})^\alg}$ denotes the $F$-isocrystal associated to the algebraization of the universal $p$-divisible group restricted to $V(\frak{b})$.
\end{setting}

\noindent{\it Proof of Lemma 2.9.} By construction we have
\begin{equation}\label{formaCb}
    \Aut^\circ_\frak{b}(\mathbb{Y})=\Aut^\circ_\frak{b}(\widetilde{\mathbb{Y}})^\alg\times_{V(\frak{b})^\alg}\Spec(k).
\end{equation}
Since $\Aut^\circ_\frak{b}(\widetilde{\mathbb{Y}})^\alg$ is the perfection of $V(\frak{b})^\alg=\Spec(k[\![x_1,\dots,x_n]\!])$ this proves the lemma.
$\qed$

\section{Dieudonné theory for quasiregular semiperfect rings}\label{dieudonnesw}

In this chapter we recall some elements of Dieudonné theory due to Scholze--Weinstein \cite{scholzeweinstein} and we apply them to the setting of Tate-linear varieties. The main result we use from their paper is that the Dieudonné functor is fully faithful over $R$ when $R$ is an $f$-quasiregular semiperfect $k$-algebra.

\begin{thm}[{\cite[Theorem A]{scholzeweinstein}}]\label{scholzeweinstein}
    The Dieudonné functor, which maps a $p$-divisible group $\mathbb{X}$ over $R$ to the Dieudonné crystal $\mathbb{D}(\mathbb{X})$ (or $\mathbb{D}^\vee(\mathbb{X})$), is fully faithful.
\end{thm}

Semiperfect rings $R$ have a universal PD-thickening $\Acr(R)\onto{} R$. Thus a contravariant Dieudonné crystal over $R$ is the same as an $\Acr(R)$-module $N$ together with maps
\begin{equation*}
    V:N\otimes_{\Acr(R),F}\Acr(R)\to N,\quad F:N\to N\otimes_{\Acr(R),F}\Acr(R)
\end{equation*}
satisfying $FV=VF=p$. Covariant Dieudonné crystals admit a similar description. For example, $\mathbb{D}^\vee((\Qp/\Zp)_R)$ is the free $\Acr(R)$-module on one generator $e$ satisfying $F(e)=e$. We will often consider $R$ of the form $R=T[a_1^{\pmeninf},\dots,a_n^{\pmeninf}]/(a_1,\dots,a_n)$, where $T$ is a perfect $\Fp$-algebra. We then have the following formula for the perfection:
\begin{equation}\label{calcoloperf}
    R^\flat\simeq T[\![a_1^{\pmeninf},\dots,a_n^{\pmeninf}]\!],
\end{equation}
and $\Acr(R)$ is the $p$-completion of the ring
\begin{equation}\label{calcoloacris}
    W(T[\![a_1^{\pmeninf},\dots,a_n^{\pmeninf}]\!])\Big[\frac{[a_i]^m}{m!},\phantom{e} m\ge0\Big],
\end{equation}
see e.g. {\cite[2.5.3-2.5.4]{drinfeldacris}}.

\subsection{A construction from Scholze--Weinstein}

Let $\mathbb{X}$ be a $p$-divisible group over $k$ and let $M$ be its contravariant Dieudonné module. The $p$-adic Tate module $T_p\mathbb{X}=\varprojlim\mathbb{X}[p^n]$ of Example \ref{deftpuc} can equivalently be described as the hom-sheaf
\begin{equation*}
    \sHom_k(\Qp/\Zp,\mathbb{X}).
\end{equation*}
As mentioned above, $T_p\mathbb{X}$ is representable by an $f$-quasiregular semiperfect $k$-algebra which we denote by $C_\mathbb{X}$, which is also a cocommutative Hopf algebra. By functoriality, $\Acr(C_\mathbb{X})$ inherits the structure of a cocommutative formal Hopf algebra. We write $\Delta$ for the comultiplication of eithre Hopf algebra, which should not cause any confusion.

The following consequence of Theorem \ref{scholzeweinstein} is also discussed in \cite[Theorem 3.3.2]{drinfeldacris}, where it is deduced from classical Dieudonné theory. 

\begin{thm}\label{functorTp}
   If $R$ is an $f$-quasiregular semiperfect $k$-algebra, there exists a canonical isomorphism
    \begin{equation}\label{drinfeldmap}
        T_p\mathbb{X}(R)\simeq\Hom_{\W(k)[F]}(M,\Acr(R)).
    \end{equation}
   This isomorphism is natural in $\mathbb{X}$ and $R$. 
\end{thm}
\noindent{\it Proof.} We have
\begin{equation*}
    T_p\mathbb{X}(R)=\Hom_R(\Qp/\Zp,\mathbb{X}_R)\simeq\Hom_R(\mathbb{D}^\vee(\mathbb{X}_R),\mathbb{D}^\vee(\Qp/\Zp))
\end{equation*}
where the second isomorphism uses Theorem \ref{scholzeweinstein}. Now, $\mathbb{D}^\vee(\Qp/\Zp)\simeq\Acr(R)$ with $F(1)=1$, whereas $\mathbb{D}^\vee(\mathbb{X}_R)\simeq M\otimes_{\W(k)}\Acr(R)$. It follows that
\begin{align*}
    \Hom_R(\mathbb{D}^\vee(\mathbb{X}_R),\mathbb{D}^\vee(\Qp/\Zp))&\simeq\Hom_{\Acr(R)[F]}(M\otimes_{\W(k)}\Acr(R),\Acr(R))\\
    &\simeq\Hom_{\W(k)[F]}(M,\Acr(R)).
\end{align*}
$\hfill\square$

After inverting $p$, we also find a natural isomorphism
\begin{equation}\label{mappacoverpdiv}
    \widetilde{\mathbb{X}}(R)\simeq\Hom_{K[F]}(M[p^{-1}],\Acr(R)[p^{-1}]),
\end{equation}
which will be useful later on.

\begin{lem}\label{deftheta}
    There is a canonical map
    \begin{equation*}
        \theta_M:M\to\Acr(C_\mathbb{X})
    \end{equation*}
    of $\W(k)[F]$-modules, which is natural in $M$ and injective.
\end{lem}
\noindent{\it Proof.} The image of the identity morphism $\textnormal{id}_{T_p\mathbb{X}}\in T_p\mathbb{X}(C_\mathbb{X})$ via \eqref{drinfeldmap} defines a map from $M$ to $\Acr(C_\mathbb{X})$ which we call $\theta_M$. By design it is a canonical map of $\W(k)[F]$-modules which is natural in $M$. Moreover, $\theta_M=0$ if and only if $M=0$, because it corresponds to the universal identity morphism of $T_p(C_{\mathbb{X}})$. 

Let $N=\ker(\theta_M)$. It is a Dieudonné sub-module of $M$, so it corresponds to a $p$-divisible group $\mathbb{H}$ such that $T_p\mathbb{X}\to T_p\mathbb{H}$ is faithfully flat. It follows that $C_{\mathbb{H}}\to C_{\mathbb{X}}$ is injective, and by the description \eqref{calcoloacris} the same holds for $\Acr(C_{\mathbb{H}})\to\Acr(C_{\mathbb{X}})$. By naturality it follows that $\theta_N=0$. Therefore $N=0$, which means that $\theta_M$ is injective.
\hfill$\square$\vspace{.5em}

It turns out that the image of $\theta_M$ coincides with the set of primitive elements of the formal Hopf algebra $\Acr(C_\mathbb{X})$. The elements in the image of $\theta_M$ are thus a sorts of ``universal logarithms'' for $\mathbb{X}$, compare \cite[V]{katzjacobi}. The easy inclusion is proved in the next lemma, and it is the only one we need in this paper. For completeness, we include a proof of the other inclusion relying on auxiliary results.

\begin{lem}\label{primitivi}
    Let $S(M)$ denote the set of primitive elements of the formal Hopf algebra $\Acr(C_\mathbb{X})$. In other words, $S(M)$ is the kernel of
    \[\begin{tikzcd}
	{\Acr(C_\mathbb{X})} & {\Acr(C_\mathbb{X}\otimes_kC_\mathbb{X})=\Acr(C_\mathbb{X})\widehat{\otimes}_{\W(k)}\Acr(C_\mathbb{X})}
	\arrow["\begin{array}{c} p_1^*+p_2^*-\Delta \end{array}", from=1-1, to=1-2].
    \end{tikzcd}\]
    Then the image of $\theta_M$ is contained in $S(M)$.
\end{lem}
\noindent{\it Proof.} Via Dieudonné theory, the sum homomorphism $+:\mathbb{X}\times \mathbb{X}\to \mathbb{X}$ corresponds to the diagonal inclusion $M\to M\oplus M$. Therefore, by naturality of $\theta$ the diagram
\[\begin{tikzcd}
	{M} & {M\oplus M} \\
	{\Acr(C_\mathbb{X})} & {\Acr(C_\mathbb{X}\otimes_k C_\mathbb{X})}
	\arrow[from=1-1, to=1-2]
	\arrow["{\theta(M)}"', from=1-1, to=2-1]
	\arrow["{\theta_{M\oplus M}}", from=1-2, to=2-2]
	\arrow["{\Delta}", from=2-1, to=2-2]
\end{tikzcd}\]
commutes. To see the inclusion $\im(\theta_M)\subseteq S(M)$ it is enough to check that $\theta_{M\oplus M}$ maps  $(x,y)\in M\oplus M$ to $\theta_M(x)\otimes1+1\otimes\theta_M(y)$, which follows once more from the naturality of $\theta$.
$\qed$

\begin{prop}\label{isomsm}
    The image of $\theta_M$ is equal to $S(M)$.
\end{prop}

\noindent{\it Proof.}
Let $\mathbb{Y}$ be another $p$-divisible groups with contravariant Dieudonné module $N$. Inside the group 
$\Hom(T_p\mathbb{X},T_p\mathbb{Y})=T_p\mathbb{Y}(C_{\mathbb{X}})$ lies the subgroup $\Hom_{\text{grp}}(T_p\mathbb{X},T_p\mathbb{Y})\simeq\Hom_{\text{grp}}(\mathbb{X},\mathbb{Y})$. Using the compatibility of \eqref{drinfeldmap} with products, one sees that this subgroup corresponds to the subset of all $f\in\Hom_{\W(k)[F]}(N,\Acr(C_{\mathbb{X}}))$ such that $f(N)$ lands in the set of primitive elements of $\Acr(C_{\mathbb{X}})$. We thus have the following commutative diagram:
\[\begin{tikzcd}
	{} & {\Hom(T_p\mathbb{X},T_p\mathbb{Y})} & {\Hom_{\W(k)[F]}(N,\Acr(C_{\mathbb{X}}))} & \\
	{} & {\Hom_{\text{grp}}(T_p\mathbb{X},T_p\mathbb{Y})} & {\Hom_{\W(k)[F]}(N,S(M))} & {\Hom_{\W(k)[F]}(N,M)}
	\arrow["\cong", from=1-2, to=1-3]
	\arrow[hook, from=2-2, to=1-2]
	\arrow["\cong", from=2-2, to=2-3]
	\arrow[hook, from=2-3, to=1-3]
	\arrow["i"', hook', from=2-4, to=2-3]
\end{tikzcd}\]
Where the map $i$ is induced by $\theta_M:M\to S(M)$. Unraveling the definitions, one checks that the map
\begin{equation*}
    {\Hom_{\W(k)[F]}(N,M)}\to{\Hom_{\text{grp}}(T_p\mathbb{X},T_p\mathbb{Y})}
\end{equation*}
obtained by composing $i$ with the reciprocal of the lower horizontal isomorphism is the classical isomorphism coming from Dieudonné theory. Therefore $i$ is an isomorphism. If we show that $S(M)$ is a Dieudonné module, then Yoneda's lemma will imply that $\theta_M$ is an isomorphism.

Firstly note that if $f:T_p\mathbb{X}\to T_p\mathbb{X}_1$ is a group homomorphism, then $f^*(S(M))\subseteq S(M)$. Thus the Verschiebung $V$ acts on $S(M)$, and one has $pS=F(V^*(S(M)))\subseteq F(S(M))$. It remains to prove that $S(M)$ is a $\W(k)$-module of finite type, and to do this we use some results from \cite{mondalstacks}, from which we also borrow conventions on stacks and derived crystalline cohomology.

Since $T_p\mathbb{X}$ is quasiregular semiperfect, the crystalline cohomology of $T_p\mathbb{X}$ and of its powers is concentrated in degree $0$, and by the crystalline Künneth formula it isomorphic to completed tensor powers of $\Acr(C_\mathbb{X})$. Using descent for crystalline cohomology along the quasisyntomic cover $*\to B(T_p\mathbb{X})$, one finds that the crystalline cohomology of $B(T_p\mathbb{X})$ is computed by a complex whose terms are precisely these completed tensor powers, and an explicit computation shows that $\Hcris{1}(B(T_p\mathbb{X}))$ is identified with the kernel of the map defining $S(M)$. Therefore $\Hcris{1}(B(T_p\mathbb{X}))$ is isomorphic to $S(M)$, and it is enough to show that the former is a finitely generated $\W(k)$-module.

On the other hand, since crystalline cohomology coincides with derived crystalline cohomology for quasisyntomic schemes \cite[8.17]{bhattlurie'}, which commutes with colimits, there is a canonical isomorphism
\begin{equation}\label{cristp}
    R\Gamma_\textnormal{cris}(B(T_p\mathbb{X}))=R\Gamma_\textnormal{cris}(\varprojlim B\mathbb{X}[p^n])\simeq\big(\varinjlim R\Gamma_\textnormal{cris}(B\mathbb{X}[p^n])\big)^\wedge_p.
\end{equation}
By \cite[2.1.8]{mondalstacks} the group $\Hcris{1}(B\mathbb{X}[p^n])$ vanishes, so \eqref{cristp} reduces to
\begin{equation*}
    \Hcris{1}(B(T_p\mathbb{X}))\simeq T_p\big(\varinjlim\Hcris{2}(B\mathbb{X}[p^n])\big).
\end{equation*}
By \cite[2.2.3]{mondalstacks}, if $G$ is a finite commutative $p$-group scheme over $k$ then $\Hcris{2}(BG)$ is canonically isomorphic to the contravariant Dieudonné module of $G$. Therefore $\Hcris{1}(T_p\mathbb{X})$ is isomorphic to $M$, and we are done. 
$\qed$

\begin{rem}
    In the proof of Proposition \ref{isomsm} we deduce an isomorphism $M\simeq S(M)$ from the results of \cite{mondalstacks}. It is in fact canonical, but it does not seem straightforward to show that it is given by $\theta_M$. It is also possible to show the finiteness of $S(M)$ using only the weaker finiteness results for $\Hcris{2}(BG)$ proved in \cite{mondalstacks}.
\end{rem}

\begin{ex}
    When $\mathbb{X}=\varinjlim_n\mu_{p^n}$, we have $C_\mathbb{X}\simeq k[x^{\pmeninf}]/(x-1)$ with comultiplication given by 
    $$\Delta(x^{p^{-i}})=x^{p^{-i}}\otimes x^{p^{-i}}.$$
    By \cite[2.5.4]{drinfeldacris} we can identify $\Acr(C_\mathbb{X})$ with the ring of 
    power series of the form
    \begin{equation*}
        \sum_{\alpha\in\Z[p^{-1}]_{\ge0}}b_\alpha\frac{[x_0]^{\alpha}}{(\alpha!)_p},
    \end{equation*}
    where the $b_\alpha\in\W(k)$ converge $p$-adically to $0$ (here $[x_0]$ is the Teichmuller lift of $x_0=(x^{1/p},x^{1/p^2},\dots)\in C_{\mathbb{X}}^\flat$). It easy to check that $\log([x_0])$ is a primitive element in $\Acr(C_\mathbb{X})$ and that $F(\log([x_0]))=p\log([x_0])$, so it is a generator of the Dieudonné module of $\mathbb{X}$ inside $\Acr(C_\mathbb{X})$.
\end{ex}

\begin{lem}\label{iniettività}
    The map of rings
    \begin{equation}\label{mapparg}
    \theta_M^a:\Sym_{\W(k)}(M)\to\Acr(C_\mathbb{X})
    \end{equation}
    induced by $\theta_M$ is injective.
\end{lem}
\noindent{\it Proof.} If we endow $\Sym_{\W(k)}(M)$ with the Hopf algebra structure of a vector group, Lemma \ref{primitivi} tells us that $\theta_M^a$ is a map of Hopf algebras. Therefore $I:=\ker(\theta_M^a)$ is a Hopf ideal of $\Sym_{\W(k)}(M)$.

We claim that $I=0$. It is enough to check this after inverting $p$, and then $I_{K}\subseteq\Sym_{K}(M[p^{-1}])$ is a Hopf ideal defining a closed subgroup of $\mathbb{G}_{a,K}^r$. In characteristic zero, any subgroup of a vector group is a vector subgroup. Therefore $I_{\Qp}$ is generated by a $K$-submodule of $M[p^{-1}]$. Since $\theta_M$ is injective, we have $I_{K}=0$, thus also $I=0$.
$\qed$

\begin{rem}\label{trivialintrem}
    If $\mathbb{X}$ is connected the subalgebra $W(C_\mathbb{X}^\flat)\sub\Acr(C_\mathbb{X})$ is a formal Hopf subalgebra which has trivial intersection with the image of $\theta^a_M$ (\textit{i.e.} the intersection is equal to $\W(k)$). Indeed, $F$ is an isomorphism on $W(C_\mathbb{X}^\flat)$, whereas $F$ is topologically nilpotent on any element of $\Sym_{W(k)}(M)$ with no constant term.
\end{rem}

\subsection{The action of $\Aut_\frak{b}^\circ(\mathbb{Y})^\Delta$ on $\mathbb{D}(\mathbb{Y})$}

Taking $\mathbb{X}=\mathcal{H}_\mathbb{Y}^\circ$, we can connect formula \eqref{drinfeldmap} with the isomorphism $T_p\mathcal{H}_\mathbb{Y}^\circ\simeq\End^\circ(\mathbb{Y})$.

\begin{lem}\label{taut1}
    Let $R$ be $f$-quasiregular semiperfect. The composition of isomorphisms
    \begin{equation*}
        \End^\circ(\mathbb{Y})(R)\simeq T_p\mathcal{H}_\mathbb{Y}^\circ(R)\simeq\Hom_{\W(k)[F]}((\frak{a}_\mathbb{Y}^+)^\vee,\Acr(R))
    \end{equation*}
    is given by the following procedure. A map $f\in\End^\circ(\mathbb{Y})(R)$ corresponds by Dieudonné theory to an $\Acr(R)[F]$-linear map
    \begin{equation*}
        \mathbb{D}(\mathbb{Y})\otimes\Acr(R)\to\mathbb{D}(\mathbb{Y})\otimes_{\W(k)}\Acr(R),
    \end{equation*}
    thus to a $\W(k)[F]$-linear map
    \begin{equation*}
        \mathbb{D}(\mathbb{Y})\to\mathbb{D}(\mathbb{Y})\otimes_{\W(k)}\Acr(R),
    \end{equation*}
    and finally to a $\W(k)[F]$-linear map
    \begin{equation*}
        \textnormal{End}(\mathbb{D}(\mathbb{Y}))^\vee\to\Acr(R).
    \end{equation*}
    The latter factors through the quotient $\textnormal{End}(\mathbb{D}(\mathbb{Y}))^\vee\onto{}(\frak{a}_\mathbb{Y}^+)^\vee$ to give the desired map.
\end{lem}
\noindent{\it Proof.}
This is clear from the description of $\mathfrak{a}_\mathbb{Y}^+$ as a submodule of $\textnormal{End}(\mathbb{D}(\mathbb{Y}))$ given in Proposition \ref{dieudonneinternalhom}.
$\qed$

\begin{rem}
    An endomorphism $f\in\End(\mathbb{Y})(R)$ corresponds to a unique $\W(k)[F]$-linear map $\textnormal{End}(\mathbb{D}(\mathbb{Y}))^\vee\to\Acr(R)$, which is the same as an $F$-invariant element of $\textnormal{End}(\mathbb{D}(\mathbb{Y}))\otimes_{\W(k)}\Acr(R)$. Choosing a basis of $\mathbb{D}(\mathbb{Y})$, the latter can be seen as an $\Acr(R)$-valued matrix. Lemma \ref{taut1} says that $f\in\End(\mathbb{Y})(R)$ belongs to $\End^\circ(\mathbb{Y})(R)$ if and only if its matrix belongs to the matrix subspace $\frak{a}_\mathbb{Y}^+\otimes_{\W(k)}\Acr(R)\subseteq\textnormal{End}(\mathbb{D}(\mathbb{Y}))\otimes_{\W(k)}\Acr(R)$.
\end{rem}

In the next Lemma we describe the exponential map of \eqref{exp} in terms of matrices. Since the exponential is defined only on the universal cover $\widetilde{\mathcal{H}}_\mathbb{Y}^\circ(R)$, we have to invert $p$ in the statement. Transposing the above discussion on $T_p\mathcal{H}^\circ_\mathbb{Y}(R)$ to $\widetilde{\mathcal{H}}_\mathbb{Y}^\circ(R)$ is easy and left to the reader.

\begin{lem}\label{taut2}
    Let $f\in\widetilde{\mathcal{H}}^\circ_\mathbb{Y}(R)=\End^\circ(\mathbb{Y})[p^{-1}]$, and consider $\exp(f)\in\Aut^\circ_{\widetilde{\mathbb{Y}}}(R)$. The matrix of $\exp(f)$ is equal to the exponential of the matrix of $f$. 
\end{lem}
\noindent{\it Proof.}
It is enough to show that for $f,g\in\End(\widetilde{\mathbb{Y}})(R)$ the matrix of $f+g$ is the sum of the matrices of $f$ and $g$, and that the matrix of $f\circ g$ is the product of the matrices of $f$ and $g$. This follows from functoriality of the Dieudonné functor. 
$\qed$

Using this point of view, we can also describe the value of $\Aut^\circ(\mathbb{Y})$ on a quasiregular semiperfect $k$-algebra $R$. Just as before, an automorphism $f\in\Aut(\mathbb{Y})(R)$ corresponds to a $\W(k)[F]$-linear map
\begin{equation*}
    \textnormal{End}(\mathbb{D}(\mathbb{Y}))^\vee\to\Acr(R).
\end{equation*}
The condition that $f$ is an automorphism translates into the fact that the map
\begin{equation}\label{nosso}
    \Sym_{\W(k)}(\textnormal{End}(\mathbb{D}(\mathbb{Y}))^\vee)\to\Acr(R)
\end{equation}
factors through the localization $R_\mathbb{Y}^+$ (see \eqref{anellideigruppi} for notation).
\begin{lem}\label{funtoreaut}
    The automorphism $f$ lies in $\Aut^\circ(\mathbb{Y})(R)$ if and only if the map \eqref{nosso} factors through the surjection $R^+_\mathbb{Y}\onto{}R^+_{\frak{a}_\mathbb{Y}}$. In particular, we have isomorphisms
    \begin{equation*}
        \Aut^\circ_\mathbb{Y}(R)\simeq\Hom_{\W(k)[F]\textnormal{-alg}}(R^+_{\frak{a}_\mathbb{Y}},\Acr(R))
    \end{equation*}
    which are natural in $R$.
\end{lem}
\noindent{\it Proof.}
Note that $f\in\Aut^\circ(\mathbb{Y})(R)$ if and only if $f=\id_{\mathbb{Y}_R}+g$ with $g\in\End^\circ(\mathbb{Y})(R)$. On the other hand, since $\frak{a}_\mathbb{Y}$ is closed under composition, we have $U(\frak{a}_\mathbb{Y})(S)=\{1+f| f\in\frak{a}_\mathbb{Y}\otimes_K S\}$ for all $K$-algebras $S$. Therefore the statement follows from Lemma \ref{taut1}.
$\qed$

Now we come to the main result of the subsection: let $\frak{b}\subseteq\frak{a}_\mathbb{Y}$ be a Dieudonné-Lie subalgebra. The natural action of $\Aut^\circ_\frak{b}(\mathbb{Y})=\Spec(C_\frak{b})$ on $\mathbb{Y}$ gives rise, via Dieudonné theory, to a \textit{comodule} map $\mathbb{D}(\mathbb{Y})\to\mathbb{D}(\mathbb{Y})\otimes_{\W(k)}\Acr(C_\frak{b})$. This map is Frobenius-equivariant and becomes the identity after composing with the augmentation map $\Acr(C_\frak{b})\to \W(k)$. Arguing as above, this corresponds to a morphism of (formal) Hopf algebras $R_\mathbb{Y}^+\to\Acr(C_\frak{b})$ which factors through $R_\mathbb{Y}^+\onto{}R_{\frak{a}_\mathbb{Y}}^+$, and we thus get a map of (formal) Hopf algebras $\sigma_\frak{b}:R^+_{\frak{a}_\mathbb{Y}}\to\Acr(C_\frak{b})$.

\begin{lem}\label{lemmafinalehopf}
    The map $\sigma_\frak{b}$ factors through the natural surjection $R_{\frak{a}_\mathbb{Y}}^+\onto{}{}R_\frak{b}^+$. The resulting morphism of Hopf algebras
    \begin{equation*}
        \rho_\frak{b}^*:R_\frak{b}^+\to\Acr(C_\frak{b}),
    \end{equation*}
    is injective.
\end{lem}
\noindent{\it Proof.} It is enough to show this after inverting $p$. Write $\exp^*:R_{\frak{a}_\mathbb{Y}}\to\Sym_K(\frak{b}^\vee)$ for the map induced by the exponential map for $U(\frak{b})$ (see \eqref{anelloub}). The exponential $\exp:\widetilde{\mathcal{H}}_{\mathbb{Y},\frak{b}}^\circ\simeq\Aut_\frak{b}^\circ(\widetilde{\mathbb{Y}})$ of \eqref{exp} induces an isomorphism $\Acr(C_{\mathcal{H}^\circ_{\mathbb{Y},\frak{b}}})[p^{-1}]\simeq \Acr(C_\frak{b})[p^{-1}]$, and we claim that the following diagram commutes:
\begin{equation}\label{diagrammafinale}
\begin{tikzcd}
	{\Sym_K(\frak{b}^\vee)} & {\Acr(C_{\mathcal{H}^\circ_{\mathbb{Y},\frak{b}}})[p^{-1}]} \\
	{R_{\frak{a}_\mathbb{Y}}} & {\Acr(C_\frak{b})[p^{-1}]}
	\arrow[from=1-1, to=1-2]
	\arrow["{\exp^*}", from=2-1, to=1-1]
	\arrow[from=2-1, to=2-2]
	\arrow["\simeq"', from=2-2, to=1-2]
\end{tikzcd}
\end{equation}
This gives the desired conclusion, because then by Lemma \ref{iniettività} the kernel of the bottom horizontal map is precisely $\ker(\exp^*)$, which is the kernel of the surjection $R_{\frak{a}_\mathbb{Y}}\onto{}R_\frak{b}$. Furthermore, since all maps considered here are natural, it is enough to prove this for $\frak{b}=\frak{a}_\mathbb{Y}$. 

Consider the following diagram for $R$ f-quasiregular semiperfect:
\[\begin{tikzcd}
	{\widetilde{\mathcal{H}}_\mathbb{Y}^\circ(R)} & {\Aut^\circ(\widetilde{\mathbb{Y}})(R)} \\
	{\Hom_{K[F]\textnormal{-alg}}(\Sym_{K}((\frak{a}_\mathbb{Y})^\vee),\Acr(R)[p^{-1}])} & {\Hom_{K[F]\textnormal{-alg}}(R_{\frak{a}_\mathbb{Y}},\Acr(R)[p^{-1}])}
	\arrow["\exp", from=1-1, to=1-2]
	\arrow["\simeq"', from=1-1, to=2-1]
	\arrow["\simeq"', from=1-2, to=2-2]
	\arrow["{\varepsilon'}", from=2-1, to=2-2]
\end{tikzcd}\]
The left-hand homomorphism comes from the identification
\begin{equation*}
    \Hom_{K[F]}((\frak{a}_\mathbb{Y})^\vee,\Acr(R)[p^{-1}])\simeq\Hom_{K[F]\textnormal{-alg}}(\Sym_{K}((\frak{a}_\mathbb{Y})^\vee),\Acr(R)[p^{-1}]),
\end{equation*}
the right-hand isomorphism is Lemma \ref{funtoreaut} and the map on the bottom is induced by matrix exponentiation $R_{\frak{a}_\mathbb{Y}}\simeq\Sym_K((\frak{a}_\mathbb{Y})^\vee)$ and is the same map as \eqref{anelloub}. By Lemma \ref{taut2} it is commutative. Taking $R=C_\frak{b}$ and tracking the map $\widetilde{\mathcal{H}}_\mathbb{Y}^\circ$ around the diagram, we see that \ref{diagrammafinale} is commutative. 
$\qed$

\begin{rem}\label{thetarapp} We can interpret $\rho_\frak{b}^*$ as a homomorphism of (formal) group schemes
\begin{equation*}
    \rho_\frak{b}:\Aut_\frak{b}^\circ(\mathbb{Y})^\Delta\to U(\frak{b})_\textnormal{int}
\end{equation*}
which corresponds to the canonical representation of $\Aut_\frak{b}^\circ(\mathbb{Y})^\Delta$ acting on $\mathbb{D}(\mathbb{Y})$ after composing with the closed immersion $U(\frak{b})_\textnormal{int}\subseteq\GL(\mathbb{D}(\mathbb{Y}))$. If $\frak{c}\subseteq\frak{b}$ is a Dieudonné-Lie subalgebra, there is a commutative diagram of (formal) Hopf algebras
\[\begin{tikzcd}
	{R_\frak{b}^+} & {\Acr(C_\frak{b})} \\
	{R_\frak{c}^+} & {\Acr(C_\frak{c})}
	\arrow["{\rho_\frak{b}^*}", from=1-1, to=1-2]
	\arrow[two heads, from=1-1, to=2-1]
	\arrow[from=1-2, to=2-2]
	\arrow["{\rho_\frak{c}^*}", from=2-1, to=2-2]
\end{tikzcd}\]
\end{rem}

\section{Monodromy and prismatization}\label{mainsec}

We sketched in the introduction how the formalism of prismatization can be used to compute $p$-adic monodromy groups. In this section we develop this idea in a special case, which is roughly speaking the quasisyntomic analogue of a Galois cover. The key example to keep in mind is the cover $\Aut^\circ_\frak{b}(\widetilde{\mathbb{Y}})^\alg\to V(\frak{b})^\alg$ of Setting \ref{setting}, which is a torsor for the quasisyntomic group scheme $\Aut^\circ_\frak{b}(\mathbb{Y})$.

Our constructions rely on the prismatization functor $Z\mapsto Z^\Delta$ of Bhatt--Lurie and Drinfeld, see e.g. $\cite{bhattlurie'}$ for a detailed account. For the purposes of this article, it is enough to know the following fundamental properties of the prismatization of a quasisyntomic scheme $Z$\footnote{See \cite[2.3.2]{marcopol} for the definition of quasisyntomic. Frobenius-smooth schemes and $f$-quasiregular semiperfect schemes are examples of quasisyntomic schemes.}. 
\begin{enumerate}
    \item The prismatization $Z^\Delta$ is a $p$-adic formal stack. If $Z=\Spec(R)$ and $R$ is $f$-quasiregular semiperfect, then $Z^\Delta\simeq\Spf(\Acr(R))$ (which is a $p$-adic formal scheme) \cite[6.1]{bhattlurie'}.
    \item The category $\textnormal{QCoh}(Z^\Delta)$ of quasicoherent sheaves on $Z^\Delta$ is equivalent to the category of quasi-coherent crystals on the (absolute) crystalline site of $Z$ \cite[8.15]{bhattlurie'}. Under this equivalence, the subcategory $\textnormal{Vect}(Z^\Delta)$ of vector bundles correspond to finite rank locally free coherent crystals (see e.g. \cite[4.7]{bhattscholzecrys}).
    \item If $Z_1\to Z_2$ is a quasisyntomic cover, then $Z_1^\Delta\to Z_2^\Delta$ is an fpqc cover \cite[7.5]{bhattlurie'}.
\end{enumerate}

\subsection{$p$-adic monodromy groups in a special situation}\label{sottomon}

In this subsection, we consider a $G$-torsor $\widetilde{Z}\to Z$ in the category of $k$-schemes, satisfying the following conditions:
\begin{itemize}
    \item The scheme $Z$ is affine, irreducible, Noetherian and Frobenius-smooth, with coordinate ring $A$ and field of constants $k$. We assume that it has a closed point $\eta:\Spec(k)\to Z$.
    \item The scheme $\widetilde{Z}$ is affine, of the form $\Spec(B)$ where $B$ is $f$-quasiregular semiperfect.
    \item The group scheme $G$ is affine, of the form $\Spec(C)$ where $C$ is $f$-quasiregular semiperfect.
\end{itemize}

As recalled above, we have $\widetilde{Z}^\Delta=\Spf(\Acr(B))$ and $G^\Delta=\Spf(\Acr(C))$, which are $p$-adic formal schemes over $\W(k)$. Moreover, $\Acr(C)$ carries a formal comultiplication induced by the comultiplication of $C$, making $G^\Delta$ into a $p$-adic formal group, and the action of $G$ on $\widetilde{Z}$ carries over to an action of $G^\Delta$ on $\widetilde{Z}^\Delta$, which makes $\widetilde{Z}^\Delta\to Z^\Delta$ into a $G^\Delta$-torsor \cite[6.3.6]{marcopol}. We write $\rho_B:\Acr(B)\to\Acr(B)\widehat{\otimes}_{\W(k)}\Acr(C)$ for the comodule map corresponding to the action of $G^\Delta$ on $\widetilde{Z}^\Delta$.

\begin{lem}\label{discesa}  
    The natural pullback functor
    \begin{equation*}
        \textnormal{QCoh}(Z^\Delta)\to\textnormal{QCoh}^{G^\Delta}(\widetilde{Z}^\Delta),
    \end{equation*}
    is an equivalence, with the category of $G^\Delta$-equivariant sheaves on $\widetilde{Z}^\Delta$ on the right-hand side. It restricts to an equivalence between vector bundles on $Z^\Delta$ and $G^\Delta$-equivariant vector bundles on $\widetilde{Z}^\Delta$.
\end{lem}
\noindent{\it Proof.} For a quasisyntomic $k$-scheme $Y$ denote by $Y_n^\Delta$ the base-change of $Y^\Delta$ to $\Spec(\W_n(k))$. The map $\widetilde{Z}^\Delta_n\to Z^\Delta_n$ is a $G^\Delta_n$-torsor, and an fpqc cover by property 3 above. Since $\widetilde{Z}^\Delta_n$ and $G^\Delta_n$ are actual schemes, we may write $Z^\Delta_n=[\widetilde{Z}^\Delta_n/G^\Delta_n]$. Therefore
\begin{equation}\label{qcohn}
    \textnormal{QCoh}(Z^\Delta_n)\simeq\textnormal{QCoh}([\widetilde{Z}^\Delta_n/G^\Delta_n])=\textnormal{QCoh}^{G_n^\Delta}(\widetilde{Z}_n^\Delta).
\end{equation}
By definition, if $A$ is a flat $p$-adically complete $\W(k)$-algebra, a quasicoherent sheaf on $\Spf(A)$ is a compatible system of quasicoherent sheaves on $\Spec(A/p^n)$. Hence taking the limit in \eqref{qcohn} we have the first equivalence (see also the remark after Theorem 2.2 in \cite{bhattscholzecrys}). The second equivalence is proved in the same way.
$\qed$

We apply this remark to a particularly nice class of $G^\Delta$-equivariant vector bundles on $\widetilde{Z}^\Delta$, given by the representations of $G^\Delta$. A representation of $G^\Delta$ in this context is a $G^\Delta$-equivariant vector bundle on $\Spf(\W(k))$, endowed with the trivial action. More concretely it is a finite rank free $\W(k)$-module $V$, with a comodule structure $V\to V\otimes_{\W(k)}\Acr(C)$. We can define a $\W(k)$-linear $\otimes$-functor
\begin{equation*}
    \mathcal{R}:\Rep(G^\Delta)\to\textnormal{Vect}^{G^\Delta}(\widetilde{Z}^\Delta),
\end{equation*}
as the pull-back along $\widetilde{Z}^\Delta\to\Spf(\W(k))$. We let $\mathcal{R}_{\Q}$ denote the composition
\begin{equation*}
    \mathcal{R}_{\Q}:\Rep(G^\Delta)\otimes\mathbb{Q}\to\textnormal{Vec}^{G^\Delta}(\widetilde{Z}^\Delta)\otimes\Q\simeq\textnormal{Vec}(Z^\Delta)\otimes\Q\to\textnormal{Isoc}(Z),
\end{equation*}
where the middle equivalence is given by Lemma \ref{discesa}, and the last map is the natural, fully faithful inclusion given by property 2 of the prismatisation. By construction it is a $K$-linear $\otimes$-functor $R_\mathbb{Q}$, and hence a functor of $K$-linear neutral Tannakian categories. Given a representation $V$, we denote $\mathcal{R}_{\Q}(V)$ the corresponding isocrystal.

The Tannakian group of $\langle V\rangle^\otimes$ can be described explicitly in terms of the comodule map $\rho:V\to V\otimes_{\W(k)}\Acr(C)$. 

\begin{deffa}\label{defav}
    Denote by $A(V)$ the smallest $p$-adic formal Hopf subalgebra of $\Acr(C)$ such that $\rho$ factors through $V\otimes A(V)$. 
\end{deffa}

Since $V$ is finite-dimensional, it is an actual Hopf subalgebra of $\Acr(C)$, meaning that $\Delta(A(V))\subseteq A(V)\otimes A(V)$. It can be described as follows: choosing a $\W(k)$-basis $e_1,\dots,e_n$ of $V$, we may write $\rho(e_i)=\sum e_j\otimes a_{i,j}$ with the $a_{i,j}\in\Acr(C)$ satisfying $\Delta(a_{i,j})=\sum_k a_{i,k}\otimes a_{k,j}$. Then $A(V)$ is the $\W(k)$-subalgebra of $\Acr(C)$ generated by the $a_{i,j}$ and their antipodes $a_{i,j}^\vee$. In particular, $A(V)$ is a flat finitely generated $\W(k)$-algebra. Inverting $p$, we get a $K$-Hopf algebra $A(V)\otimes_{\W(k)}K$ which defines a closed subgroup $G_V$ of $\GL(V_K)$.
\begin{lem}\label{tannaka}
    We have $G(\langle V\rangle^\otimes)\simeq G_V$.
\end{lem}
\noindent{\it Proof.} If $B$ is a flat Hopf $\W(k)$-algebra (resp. a Hopf $K$-algebra $B$), denote by $\textnormal{Comod}_B$ the category of $B$-comodules which are free finitely generated over $\W(k)$ (resp. finite-dimensional over $K$).

Let $A_C$ the largest Hopf $\W(k)$-subalgebra of $\Acr(C)$, \textit{i.e.} the largest $\W(k)$-subalgebra $A\subseteq \Acr(C)$ such that $\Delta(A)\subseteq A\otimes A$ (and not merely $\Delta(A)\subseteq A\widehat{\otimes}A$), and which is closed under the antipode map. The construction $V\mapsto A(V)$ of Definition \ref{defav} shows that $\textnormal{Rep}(G^\Delta)$ is equivalent to the category $\textnormal{Comod}_{A_C}$. We claim that the obvious functor
\begin{equation*}
    \textnormal{Comod}_{A_C}\otimes\Q\to\textnormal{Comod}_{A_C\otimes K}
\end{equation*}
is an equivalence of categories. The full faithfulness is immediate, whereas the essential surjectivity follows from \cite[2.4, Lemme 2]{serrelattice}.

Thus, we are in the following situation: we have an $A_C\otimes K$ comodule $V_K$, and the smallest Hopf-subalgebra on which it is defined is $A(V)\otimes K\subseteq A_C\otimes K$. We want to conclude that the Tannakian subcategory of $\textnormal{Comod}_{A_C\otimes K}$ spanned by $V_K$ is equivalent to $\textnormal{Comod}_{A(V)\otimes K}$. This is an immediate consequence of the formula for the Tannaka comodule of $\langle V_K\rangle^\otimes$ as a coend.
$\qed$

In order to compare the Tannakian groups of $\langle V\rangle^\otimes$ and $\langle\mathcal{R}_{\Q}(V)\rangle^\otimes$, we start by determining the hom-sets $\Hom(\mathcal{R}(W_1),\mathcal{R}(W_2))$ for two representations $W_1$ and $W_2$.

\begin{lem}\label{homsets}
    Let $L=\Hom_{W(k)}(W_1,W_2)$ be the internal-hom representation. There is a $W(k)$-linear bijection
    \begin{equation*}
        \Hom(\mathcal{R}(W_1),\mathcal{R}(W_2))\simeq\Hom_{\Acr(C)\textnormal{-comod}}(L^\vee,\Acr(B)).
    \end{equation*}
    The subset $\Hom(W_1,W_2)$ corresponds to the comodule maps $L^\vee\to \Acr(B)$ which factor through the canonical inclusion $W(k)\to\Acr(B)$.
\end{lem}
\noindent{\it Proof.} This is a standard calculation with equivariant vector bundles. A morphism between $\mathcal{R}(W_1)$ and $\mathcal{R}(W_2)$ is the same as a section of the Hom vector bundle $\mathcal{R}(L)$. In turn this is the same as a a $G^\Delta$ equivariant map $\widetilde{Z}^\Delta\to\Spec(\Sym_{\W(k)}(L^\vee))$, which amounts to an $\Acr(C)$-comodule morphism $L^\vee\to\Acr(B)$. The second statement is immediate.
$\qed$

\begin{cor}\label{fullyfaithful}
    (1) Let $A(B)$ denote the smallest formal Hopf subalgebra of $\Acr(C)$ through which the comodule map $\Acr(B)\to\Acr(B)\widehat{\otimes}_{\W(k)}\Acr(C)$ factors. If $A(V)$ and $A(B)$ have trivial intersection, \textit{i.e.} if $A(V)\cap A(B)=\W(k)$, then the restriction
    \begin{equation*}
        \mathcal{R}_{\Q}:\langle V\rangle^\otimes\to\langle\mathcal{R}_{\Q}(V)\rangle^\otimes
    \end{equation*}
    is fully faithful.
    
    (2) Suppose furthermore that $G(\langle V\rangle^\otimes)$ is unipotent. Then $\mathcal{R}_{\Q}$ induces, by Tannakian duality, an isomorphism $G(\langle\mathcal{R}_{\Q}(V)\rangle^\otimes)\simeq G(\langle V\rangle^\otimes)$.
\end{cor}
\noindent{\it Proof.} Let $W_1,W_2\in\langle V\rangle^\otimes$. The comodule structure on the representation $L^\vee$ factors through $A(V)$. Since $A(V)\cap A(B)=\W(k)$, any comodule map $L^\vee\to B$ factors through $W(k)$. We conclude by Lemma \ref{homsets}.

For the second statement, note that by construction $G(\langle\mathcal{R}_{\Q}(V)\rangle^\otimes)\to G(\langle V\rangle^\otimes)$ is a closed immersion. Fully faithfullness of $\mathcal{R}_{\Q}$ implies that for every representation of $G(\langle V\rangle^\otimes)$ the $G(\langle V\rangle^\otimes)$- and $G(\langle \mathcal{R}_{\Q}(V)\rangle^\otimes)$-invariants are the same. But by Chevalley's theorem, if $U_1\subseteq U_2$ are unipotent groups, and for every representation of $U_2$ the $U_1$- and $U_2$-invariants are the same, then $U_1=U_2$ (compare \cite[3.1.1]{marcopol}).
$\qed$

We will use this criterion to show, in Proposition \ref{propiso} below, that the monodromy group $G(\mathcal{M}_{Z(\frak{b}^+)})$ is maximal and equal to $U(\frak{b})$ (Proposition \ref{coriso} in the introduction). But to prove Theorem \ref{teomonodromia} we need to look more closely at what happens when the monodromy groups don't coincide.

\begin{lem}\label{inclstretta}
    Suppose that $G(\langle\mathcal{R}_{\Q}(V)\rangle^\otimes)\subsetneq G(\langle V\rangle^\otimes)$ and that $G(\langle V\rangle^\otimes)$ is unipotent. There exists a non-trivial primitive element $g\in A(V)$ and an element $x\in\Acr(B)$ such that $\rho_B(x)=x\otimes 1+1\otimes g$.
\end{lem}
\noindent{\it Proof.} By Corollary \ref{fullyfaithful} the restriction of $\mathcal{R}_{\Q}$ to $\langle V\rangle^\otimes$ is not fully faithful. Thus, by Lemma \ref{homsets}, there exists a representation $W\in\langle V\rangle^\otimes$ with a non-constant map of $\Acr(C)$-comodules $\phi:W\to\Acr(B)$, which we may assume to be injective. 

Since $G$ acts faithfully on $\widetilde{Z}$, the invariants of $\Acr(B)$ consist only of $\W(k)$, so by injectivity of $\phi$ the $G^\Delta$-invariant subspace of $W$ has rank $1$. If $a$ is a generator, $\lambda:=\phi(a)$ lies in $\W(k)\subseteq\Acr(B)$. Now take $b\in W$ which becomes a non-zero invariant in $W/\langle a\rangle$. The comodule map of $W$ sends $b$ to $b\otimes 1+a\otimes g$, where $g\in \Acr(C)$ is primitive. Therefore $x=\phi(b)$ satisfies the desired property.
$\qed$

\subsection{Monodromy of Tate-linear subvarieties}

In this subsection, we use the notation of Setting \ref{setting}. If $S\subseteq V(\frak{b})$ is a smooth formal subscheme, we can consider the three following schemes
\begin{equation*}
    Z:=S^\alg,\quad \widetilde{Z}:=\Aut^\circ_{\frak{b}}(\widetilde{\mathbb{Y}})^\alg\times_{V(\frak{b})^\alg}S^\alg,\quad G:=\Aut_{\frak{b}}^\circ(\mathbb{Y}),
\end{equation*}
and then the $G$-torsor $\widetilde{Z}\to Z$ satisfies the conditions of Subsection \ref{sottomon}. 

Let us first take $S=V(\frak{b})$. In Remark \ref{thetarapp} we had introduced the notation $\rho_\frak{b}$ for the representation of $\Aut_\frak{b}^\circ(\mathbb{Y})^\Delta$ on $\mathbb{D}(\mathbb{Y})$ induced by the canonical action of $\Aut_\frak{b}^\circ(\mathbb{Y})$ on $\mathbb{Y}$. Applying the previous section's functor $\mathcal{R}_\mathbb{Q}$, we thus obtain an isocrystal $\mathcal{N}=\mathcal{R}_{\Q}(\rho_\frak{b})$ on $S^\alg=V(\frak{b})^\alg$.

\begin{lem}
    The isocrystals $\mathcal{N}$ and $\mathcal{M}_{V(\frak{b})^\alg}$ are isomorphic.
\end{lem}
\noindent{\it Proof.} It is enough to prove this for $\frak{b}=\frak{a}_\mathbb{Y}$. By definition, the pullback of the universal $p$-divisible group on $\defsus(\mathbb{Y})$ to $\Aut^\circ(\widetilde{\mathbb{Y}})$ is the constant $p$-divisible group $\mathbb{Y}\times_k\Aut^\circ(\widetilde{\mathbb{Y}})$, on which $\Aut^\circ(\mathbb{Y})$ acts via the diagonal action. This property is preserved by taking the algebraization and the Dieudonné functor, so the pullback of $\mathcal{M}_{V(\frak{a}_\mathbb{Y})^\alg}$ to $(\Aut^\circ(\widetilde{\mathbb{Y}})^\alg)^\Delta$ is precisely $\mathcal{R}_\Q(\rho_\frak{b})$ as an equivariant coherent sheaf. 
$\qed$

Now we allow $S$ to be an arbitary smooth formal subscheme of $V(\frak{b})$. Let us introduce the coordinate rings of our spaces: we write
\begin{equation*}
    V(\frak{b})^\alg=\Spec(A_{\frak{b}}),\quad Z=S^\alg=\Spec(A_\frak{b}/I_S),\quad G=\Spec(C_{\frak{b}}).
\end{equation*}
The ring $A_\frak{b}$ is isomorphic to $k[\![x_1,\dots,x_n]\!]$, and since $S\subseteq V(\frak{b})$ is a smooth formal subscheme, $I_S$ is generated by some regular sequence $(f_1,\dots,f_m)$. Up to a change of coordinates we may assume that $f_i=x_i$ for $1\le i\le m$. By the discussion of Setting \ref{setting} we have
\begin{equation*}
    \Aut^\circ_{\frak{b}}(\widetilde{\mathbb{Y}})^\alg=\Spec((A_{\frak{b}})_\text{perf}),\quad\widetilde{Z}
    =\Spec((A_{\frak{b}})_\text{perf}/I_S(A_{\frak{b}})_\text{perf}),
\end{equation*}
and the proof of Lemma \ref{autbeqrsp} shows that
\begin{equation*}
    \quad C_{\frak{b}}=(A_{\frak{b}})_\text{perf}/(x_1,\dots,x_n)\simeq k[x_1^{\pmeninf},\dots,x_n^{\pmeninf}]/(x_1,\dots,x_n).
\end{equation*}

\begin{lem}\label{prep1}
    (1) The $G$-equivariant closed immersions
    \begin{equation*}
        \Aut^\circ_\frak{b}(\mathbb{Y})\hookrightarrow\widetilde{Z}\hookrightarrow\Aut^\circ_{\frak{b}}(\widetilde{\mathbb{Y}})^\alg
    \end{equation*}
    give rise to $G^\Delta$-equivariant maps
    \begin{equation*}
        \Aut^\circ_\frak{b}(\mathbb{Y})^\Delta\to\widetilde{Z}^\Delta\to(\Aut^\circ_{\frak{b}}(\widetilde{\mathbb{Y}})^\alg)^\Delta.
    \end{equation*}
    In coordinates these correspond to inclusions
    \begin{equation*}
        W((A_{\frak{b}})_\textnormal{perf})\hookrightarrow\Acr((A_{\frak{b}})_\textnormal{perf}/I_S(A_{\frak{b}})_\textnormal{perf})\hookrightarrow\Acr(C_{\frak{b}})
    \end{equation*}
    which are compatible with the $\Acr(C_\frak{b})$-comodule structures.
    
    (2) The image of $W((A_{\frak{b}})_\textnormal{perf})$ inside $\Acr(C_{\frak{b}})$ is contained in $W((C_{\frak{b}})^\flat)$.

\end{lem}
\noindent{Proof.} The only assertion of (1) which requires proof is that the maps of rings are inclusions, which follows from the explicit description of the functor $\Acr$ given in \eqref{calcoloacris}. Similarly, (2) follows from \eqref{calcoloperf} since $(A_{\frak{b}})_\textnormal{perf}$ injects into $(C_{\frak{b}})^\flat=k[\![x_1^{\pmeninf},\dots,x_n^{\pmeninf}]\!]$.
$\qed$

In order to prove Theorem \ref{mainthm}, we establish in Lemma \ref{prep2} a criterion for $S\subseteq V(\frak{b})$ to factor through the inclusion $V(\frak{c})\subseteq V(\frak{b})$ for a fixed Dieudonné-Lie subalgebra $\frak{c}\subseteq\frak{b}$. The first step is to describe the ideal $I_\frak{c}\subseteq A_\frak{b}$ which defines the closed immersion $V(\frak{c})^\alg\subseteq V(\frak{b})^\alg$. 

\begin{lem}\label{tecnicofinale}
    Let $I\subseteq A_\frak{b}$ be an ideal such that $A_\frak{b}/I$ is regular, and let $I^{\pmeninf}$ denote the ideal $\varinjlim_F I$ of $(A_\frak{b})_\textnormal{perf}$. Then $I=I^{\pmeninf}\cap A_\frak{b}$.
\end{lem}
\noindent{\it Proof.}
We may choose coordinates $y_1,\dots, y_n$ such that $A_\frak{b}=k[\![y_1,\dots,y_n]\!]$ and $I=(y_1,\dots,y_\ell)$ for some $1\le\ell\le n$. Then the statement says that the intersection
\begin{equation*}
    (y_1^{\pmeninf},\dots,y_\ell^{\pmeninf})\cap k[\![y_1,\dots,y_n]\!]\subseteq\varinjlim_Fk[\![y_1,\dots,y_n]\!]
\end{equation*}
is equal to $(y_1,\dots,y_\ell)$, which is clear.
$\qed$

This reduces the problem to describing the ideal $I_\frak{c}^{\pmeninf}$, which by construction is the ideal defining the closed immersion $\Aut^\circ_\frak{c}(\widetilde{\mathbb{Y}})^\alg\subseteq\Aut^\circ_\frak{b}(\widetilde{\mathbb{Y}})^\alg$. Consider the commutative diagram 
\[\begin{tikzcd}
	{\Aut^\circ_\frak{c}(\widetilde{\mathbb{Y}})} & {\Aut^\circ_\frak{b}(\widetilde{\mathbb{Y}})} & \\
	{\widetilde{\mathcal{H}}_{\mathbb{Y},\frak{c}}^\circ} & {\widetilde{\mathcal{H}}_{\mathbb{Y},\frak{b}}^\circ} & {\widetilde{\mathbb{H}}_\frak{c}}
	\arrow[hook, from=1-1, to=1-2]
	\arrow["\exp"{description}, shift left=3, draw=none, from=1-1, to=2-1]
	\arrow["\simeq", from=2-1, to=1-1]
	\arrow[hook, from=2-1, to=2-2]
	\arrow["\simeq", from=2-2, to=1-2]
	\arrow["\exp"', draw=none, from=2-2, to=1-2]
	\arrow[from=2-2, to=2-3]
\end{tikzcd}\]
where formal group scheme $\widetilde{\mathbb{H}}_\frak{c}$ is the universal cover of the $p$-divisible group whose covariant Dieudonné module is $\frak{b}/\frak{c}$. On the algebraic side, this corresponds to the commutative diagram of topological $k$-algebras 
\[\begin{tikzcd}
	& {(C_\frak{b})^\flat} & {(C_\frak{c})^\flat} \\
	{k[\![\widetilde{I}_\frak{c}]\!]} & {(C_{\mathcal{H}_{\mathbb{Y},\frak{b}^\circ}})^\flat} & {(C_{\mathcal{H}_{\mathbb{Y},\frak{c}^\circ}})^\flat}
	\arrow[two heads, from=1-2, to=1-3]
	\arrow["\simeq"', from=1-2, to=2-2]
	\arrow["\simeq"', from=1-3, to=2-3]
	\arrow[hook, from=2-1, to=2-2]
	\arrow[two heads, from=2-2, to=2-3]
\end{tikzcd}\]
where we denote by $\widetilde{I}_\frak{c}$ the augmentation ideal of $\widetilde{\mathbb{H}}_\frak{c}$. Then, by definition, the ideal $I_\frak{c}^{\pmeninf}$ is the intersection of $(A_\frak{b})_\textnormal{perf}$ and $\widetilde{I}_\frak{c}$ inside $(C_\frak{b})^\flat$.

\begin{lem}\label{prep2}
    Let $J_\frak{c}\subseteq R_\frak{b}$ denote the Hopf ideal defining the closed immersion $U(\frak{c})\subseteq U(\frak{b})$. Thus $(\frak{b}/\frak{c})^\vee$ is mapped into $J_\frak{c}$ via the inverse of the exponential map, $\log^*:\Sym_K(\frak{b}^\vee)\simeq R_\frak{b}$. If $(\frak{b}/\frak{c})^\vee$, considered as a subset of $\rho_\frak{b}^*(J_\frak{c})\subseteq\Acr(C_\frak{b})[p^{-1}]$, is contained in $\Acr((A_\frak{b})_\textnormal{perf}/I_S(A_\frak{b})_\textnormal{perf})[p^{-1}]$, then $S$ is a subscheme of $V(\frak{c})$.
\end{lem}
\noindent{Proof.} We are assuming that in the following diagram
\[\begin{tikzcd}
	{(\frak{b}/\frak{c})^\vee} & {J_\frak{c}} & {\Acr(C_\frak{b})[p^{-1}]} \\
	&& {\Acr((A_\frak{b})_\textnormal{perf}/I_S(A_\frak{b})_\textnormal{perf})[p^{-1}]}
	\arrow["{\log^*}", from=1-1, to=1-2]
	\arrow[dashed, from=1-1, to=2-3]
	\arrow["{{\rho_\frak{b}^*}}", from=1-2, to=1-3]
	\arrow[from=2-3, to=1-3]
\end{tikzcd}\]
the dashed arrow exists, and then it is necessarily unique. By virtue of \eqref{mappacoverpdiv} and the commutativity of diagram \eqref{diagrammafinale}, it corresponds uniquely to a dashed arrow in the following diagram
\[\begin{tikzcd}
	{\Aut_\frak{b}^\circ(\mathbb{Y})} & {\widetilde{Z}} \\
	& {\widetilde{\mathbb{H}}_\frak{c}}
	\arrow[from=1-1, to=1-2]
	\arrow[from=1-1, to=2-2]
	\arrow[dashed, from=1-2, to=2-2]
\end{tikzcd}\]
where the diagonal map is the composition $\Aut^\circ_\frak{b}(\mathbb{Y})\xto{}{\log}\widetilde{\mathcal{H}}_{\mathbb{Y},\frak{b}}^\circ\onto{}\widetilde{\mathbb{H}}_\frak{c}$. Since $\widetilde{\mathbb{H}}_\frak{c}$ is perfect, if $A$ is a $k$-algebra any map $\Spec(A)\to\widetilde{\mathbb{H}}_\frak{c}$ factors through the map $\Spec(A)\to\varinjlim_F\Spec(A)$. Here the sheaf $\varinjlim_F\Spec(A)$ is the affine formal scheme $\Spf(A^\flat)$ whenever the Frobenius $F$ is a thickening, which  is the case for both schemes in the top row of the previous diagram. Thus, we find that there is a dotted arrow of topological rings:
\[\begin{tikzcd}
	{(C_\frak{b})^\flat=k[\![x_1^{\pmeninf},\dots,x_n^{\pmeninf}]\!]} & {(k[\![x_{m+1},\dots,x_n]\!])_\textnormal{perf}[\![x_1^{\pmeninf},\dots,x_m^{\pmeninf}]\!]=((A_\frak{b})_\textnormal{perf}/I_S(A_\frak{b})_\textnormal{perf})^\flat} \\
	& {k[\![\widetilde{I}_\frak{c}]\!]}
	\arrow[hook', from=1-2, to=1-1]
	\arrow[from=2-2, to=1-1]
	\arrow[dashed, from=2-2, to=1-2]
\end{tikzcd}\]
where in the top row we have used the formula for the perfection given in Lemma \eqref{calcoloperf}. Thus every element of $I_\frak{c}^{\pmeninf}$ belongs to the topologically nilpotent elements of  $((A_\frak{b})_\textnormal{perf}/I_S(A_\frak{b})_\textnormal{perf})^\flat$. By the explicit description of this ring given in the diagram, the topologically nilpotent elements which also lie in $(A_\frak{b})_\textnormal{perf}$ coincide with $I_S^{\pmeninf}$. Therefore $I_\frak{c}^{\pmeninf}\subseteq I_S^{\pmeninf}$, and by Lemma \ref{tecnicofinale} we have $I_\frak{c}\subseteq I_S$, which is what we wanted to prove.
$\qed$

Now we can reap the fruit of our work: first we take $S=V(\frak{b})$ itself and we determine the monodromy group $G(\mathcal{M}_{V(\frak{b})})$. Then we prove the stronger assertion, Theorem \ref{teomonodromia}.

\begin{prop}\label{propiso}
    The monodromy action induces an isomorphism $G(\mathcal{M}_{V(\frak{b})})\simeq U(\frak{b})$.
\end{prop}
\noindent{\it Proof.} By Lemma \ref{lemmafinalehopf}, $\rho_\frak{b}^*(R_\frak{b}^+)$ is the smallest Hopf subalgebra of $\Acr(C_\frak{b})$ through which the representation $\rho_\frak{b}$ factors. Therefore, by Lemma \ref{tannaka} we have $G(\langle{V_\frak{b}}\rangle^\otimes)=U(\frak{b})$. Furthermore,
by point $(2)$ of Lemma \ref{prep1} the smallest formal Hopf subalgebra of $\Acr(C_{\frak{b}})$ which acts on $W((A_{\frak{b}})_\text{perf})$ is contained in $W((C_{\frak{b}})^\flat)$. By Remark \ref{trivialintrem}, the hypothesis of point (2) of Corollary \ref{fullyfaithful} are verified, which implies that $G(\mathcal{M}_{V(\frak{b})})\simeq G(\langle{V_\frak{b}}\rangle^\otimes)$.
$\qed$

\begin{thm}\label{mainthm}
    Let $S\subseteq\defsus(\mathbb{Y})$ be a smooth formal subscheme. Suppose that $G(\mathcal{M}_{S})\subseteq U(\frak{b})$ for some Dieudonné-Lie subalgebra $\frak{b}\subseteq\frak{a}_Y$. Then $S\subseteq V(\frak{b})$.
\end{thm}
\noindent{\it Proof.} It is enough to show that, if $S\subseteq V(\frak{b})$ is a smooth formal subscheme and $G(\mathcal{M}_S)$ is strictly smaller than $U(\frak{b})$, then $S\subseteq V(\frak{c})$ for some strictly smaller Dieudonné-Lie algebra $\frak{c}\subsetneq\frak{b}$.

By Lemma \ref{inclstretta}, there exists a non-trivial primitive element $g\in\rho_\frak{b}^*(R_\frak{b}^+)\subseteq\Acr(C_{\frak{b}})$ and an element $x\in\Acr((A_{\frak{b}})_\text{perf}/I_S(A_{\frak{b}})_\text{perf})$ such that $\rho(x)=x\otimes 1+1\otimes g$. But then
\begin{equation*}
    \Delta(x-g)=\rho(x)-\Delta(g)=(x-g)\otimes1
\end{equation*}
where $\Delta$ is the comultiplication of $\Acr(C_\frak{b})$. In other words $x-g$ is an invariant of the regular representation of $\Acr(C_\frak{b})$, so in fact $x-g$ belongs to $\W(k)$. This implies
\begin{equation*}
    g\in \Acr((A_{\frak{b}})_\text{perf}/I_S(A_{\frak{b}})_\text{perf}).
\end{equation*}
Denote by $L$ the smallest $F$-stable $K$-submodule of $R_\frak{b}$ containing $g$. We have an inclusion $L\subseteq \Acr((A_{\frak{b}})_\text{perf}/I_S(A_{\frak{b}})_\text{perf})[p^{-1}]$, and $L$ consists primitive elements of $\rho_\frak{b}^*(R_\frak{b})$ and is $F$-stable. 

By injectivity of $\rho_\frak{b}^*$, we may consider $L$ as a subset of $R_\frak{b}$. The inclusion $\Sym_K(L)\subseteq R_\frak{b}$ then defines an epimorphism $U(\frak{b})\onto{}\mathbb{G}_a^s$. On the level of Lie algebras, it gives a surjection of Lie algebras $\frak{b}\onto{} \exp^*(L)^\vee$, which by construction is compatible with the Frobenius on both sides. Therefore, the kernel $\frak{c}$ of this surjection is a Dieudonné-Lie subalgebra of $\frak{b}$ with $\exp^*(L)=(\frak{b}/\frak{c})^\vee$.

This means that $L=\log^*((\frak{b}/\frak{c})^\vee)$ is contained in $\Acr((A_{\frak{b}})_\text{perf}/I_S(A_{\frak{b}})_\text{perf})[p^{-1}]$, and by Lemma \ref{prep2} we have $S\subseteq V(\frak{c})$, which is what we wanted to prove.
$\qed$

Finally, combining Theorem \ref{mainthm} with general results on the monodromy of isocrystals from \cite{marcopol}, we are able to prove Theorem \ref{rigiditaTL}.

\begin{thm}\label{domandaC}
    Let $C$ be a central leaf of $\calA_g$ and let $Z\subseteq C$ be a smooth, connected, locally closed subvariety. If there exists $x\in Z(k)$ such that $Z^{/x}$ is Tate-linear, then $Z^{/y}$ is Tate-linear for all $y\in Z(k)$.
\end{thm}
\noindent{\it Proof of.} By Theorem \ref{mainthm}, a smooth formal subvariety $S\subseteq C(x)^{/x}$ is Tate-linear if and only if $\dim(S)=\dim(V(\frak{b}))$, where $\frak{b}=\textnormal{Lie}(G(\mathcal{M}_Z))$. 

By \cite[3.4.4]{marcopol}, the local monodromy groups $G(\mathcal{M}_{Z^{/y}})$ are independent of the chosen point $y\in Z(k)$. Thus if $Z$ is Tate-linear at a closed point, it is Tate-linear at every closed point.
$\qed$

\vspace{.5em}
\hrule
\vspace{.5em}
\begin{center}
\begin{minipage}{5cm}
\includegraphics[width=4.8cm]{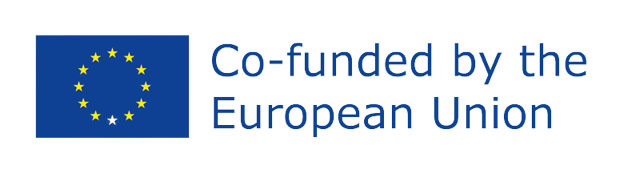}
\end{minipage}
\begin{minipage}{9.5cm}
\footnotesize Co-Funded by the European Union. Views and opinions expressed are however those of the author(s) only and do not necessarily reflect those of the European Union. Neither the European Union nor the granting authority can be held responsible for them.
\end{minipage}
\end{center}

\end{document}